\documentclass[12pt,a4paper]{article}
\usepackage{amsmath,amssymb,amsthm,graphicx,color}
\usepackage[left=3cm,right=3cm,top=1in,bottom=1in]{geometry}
\usepackage{cite}
\usepackage[british]{babel}
\usepackage{enumerate}
\usepackage{hyperref} 
\usepackage{graphicx}
\usepackage{comment,tikz}
\usepackage[font={small}]{caption}

\hypersetup{
    colorlinks=false,
    pdfborder={0 0 0},
}

\newtheorem{theorem}{Theorem}[section]
\newtheorem{corollary}[theorem]{Corollary}
\newtheorem{proposition}[theorem]{Proposition}
\newtheorem{lemma}[theorem]{Lemma}

\newtheorem*{problem}{Open Problem}

\numberwithin{equation}{section}

\theoremstyle{definition}

\newtheorem{examplex}[theorem]{Example}

\newenvironment{example}
  {\pushQED{\qed}\examplex}
  {\popQED\endexamplex}

\theoremstyle{remark}
\newtheorem{remark}[theorem]{Remark}

 \allowdisplaybreaks

\newcommand{\1}[1]{{\mathbf 1}{\{#1\}}}

\newcommand{\R}{{\mathbb R}}
\newcommand{\Z}{{\mathbb Z}}
\newcommand{\N}{{\mathbb N}}
\newcommand{\ZP}{{\mathbb Z}_+\!}
\newcommand{\RP}{{\mathbb R}_+\!}

\DeclareMathOperator{\Exp}{\mathbb{E}}
\renewcommand{\Pr}{{\mathbb P}}

\DeclareMathOperator{\sign}{sgn}

\newcommand{\tra}{{\scalebox{0.6}{$\top$}}}

\newcommand{\eps}{\varepsilon}
\newcommand{\la}{\lambda}

\newcommand{\re}{{\mathrm{e}}}
\newcommand{\rc}{{\mathrm{c}}}
\newcommand{\ud}{{\mathrm d}}
\newcommand{\cB}{{\mathcal B}}

\newcommand{\cF}{{\mathcal F}}

\newcommand{\cL}{{\mathcal L}}

\newcommand{\cO}{{\mathcal{O}}}
\newcommand{\cR}{{\mathcal R}}
\newcommand{\cS}{{\mathcal S}}

\newcommand{\as}{\ \text{a.s.}}

\newcommand{\Bigmid}{\; \Bigl| \;}

\newcommand{\ba}{{\mathbf{a}}}
\newcommand{\bb}{{\mathbf{b}}}
\newcommand{\bx}{{\mathbf{x}}}

\newcommand{\0}{{\mathbf{0}}}

\newcommand{\HG}{{}_2 F_1 }
\newcommand{\HGG}{{}_4 F_3 }

\newcommand{\fD}{{\mathfrak{D}}}

\newcommand{\Gs}{G^{\rm sym}}

\newcommand{\Rs}{R^{\rm sym}}
\newcommand{\Ro}{R^{\rm one}}

\newcommand{\So}{\cS^{\rm one}}
\newcommand{\Ss}{\cS^{\rm sym}}

\newcommand{\bla}{{\boldsymbol{\lambda}}}

\newcommand{\bmu}{{\boldsymbol{\mu}}}
\newcommand{\bt}{{\boldsymbol{\theta}}}

\newcommand{\ubar}[1]{\underline{#1\mkern-4mu}\mkern4mu }

\makeatletter
\def\namedlabel#1#2{\begingroup  
    (#2)%
    \def\@currentlabel{#2}%
    \phantomsection\label{#1}\endgroup
}
\makeatother

\begin{document}

\title{Heavy-tailed random walks\\
 on complexes of half-lines}
\author{Mikhail V.\ Menshikov\footnote{Department of Mathematical Sciences, Durham University, South Road, Durham DH1 3LE, UK.}
\and Dimitri Petritis\footnote{IRMAR, Campus de Beaulieu, 35042 Rennes Cedex, France.}
\and Andrew R.\ Wade\footnotemark[1]}

\date{4 October 2016}
\maketitle

\begin{abstract}
We study a random walk on a complex of finitely many half-lines  joined at a common origin; jumps
are heavy-tailed and of two types, either one-sided (towards the origin) or two-sided
(symmetric). Transmission between half-lines via the origin is governed by an irreducible Markov transition matrix, with associated stationary
distribution $\mu_k$.
If $\chi_k$ is $1$ for one-sided half-lines $k$ and $1/2$ for two-sided half-lines, and $\alpha_k$ is the tail exponent of the jumps on half-line $k$,
we show that the recurrence classification for the case where all $\alpha_k \chi_k \in (0,1)$
is determined by the sign of $\sum_k  \mu_k \cot ( \chi_k \pi \alpha_k )$. In the case of two half-lines,
the model fits naturally on $\R$ and is a version of the \emph{oscillating random walk} of Kemperman.
In that case, the cotangent criterion for recurrence becomes linear in $\alpha_1$ and $\alpha_2$;
our general setting exhibits the essential non-linearity in the cotangent criterion. For the general model, we also show
existence and non-existence of polynomial moments of return times. Our moments results are sharp (and new) for
several cases of the oscillating random walk; they are apparently even new for the case of a homogeneous
random walk on $\R$ with symmetric increments of tail exponent $\alpha \in (1,2)$.
\end{abstract}

\noindent
{\em Key words:}  Random walk, heavy tails, recurrence, transience, passage time moments, Lyapunov functions, oscillating random walk, cotangent criterion.

\medskip

\noindent
{\em AMS Subject Classification:} 60J05 (Primary); 60J10, 60G50  (Secondary)

\section{Introduction}
\label{sec:intro}
 
We study Markov processes on a complex of half-lines $\RP \times \cS$,
where $\cS$ is finite, and all half-lines are connected at a common origin.  
On a given half-line, a particle performs a random walk
with a heavy-tailed increment distribution, until it
would exit the half-line, when it switches (in general, at random) to another half-line to complete its jump.

To motivate the development of the general model, we first discuss informally some examples;
we give formal statements later. 

The \emph{one-sided oscillating random walk}
takes place on two half-lines, which we may map onto $\R$. From the positive half-line, the 
increments are negative with density proportional to $y^{-1-\alpha}$, and from the
negative half-line the increments are positive with density proportional to 
$y^{-1-\beta}$, where $\alpha, \beta \in (0,1)$. The walk is transient if and only if $\alpha + \beta < 1$;
this is essentially a result of Kemperman~\cite{kemperman}. 

The oscillating random walk
has several variations and has been well studied over the years (see e.g.~\cite{rf,sandric0,sandric1,sandric2}).
This previous work, as we describe in more detail below, is restricted to the case of \emph{two} half-lines.
We   generalize this model to an arbitrary number of half-lines, labelled by a finite set $\cS$,
by assigning a rule for travelling from half-line to half-line.

First we describe a deterministic rule.
Let $T$ be a positive integer. Define a \emph{routing schedule of length} $T$ to be a sequence  
 $\sigma = (i_1, \ldots, i_T)$ of $T$ elements of $\cS$, dictating the sequence in which half-lines are visited, as follows.
The walk starts from line $i_1$, and, on departure from line $i_1$ jumps over the origin to $i_2$, and so on,
until departing $i_T$ it returns to $i_1$; on line $k \in \cS$, the walk jumps towards the origin with density proportional
to $y^{-1-\alpha_k}$ where $\alpha_k \in (0,1)$. One simple example takes a cyclic schedule in which $\sigma$
is a permutation of the elements of $\cS$. In any case, a consequence of our results is that now the walk is transient
if and only if 
\begin{equation}
\label{eq:intro_criterion}
\sum_{k \in \cS} \mu_k \cot ( \pi \alpha_k ) > 0,
\end{equation}
where $\mu_k$ is the number of times $k$ appears in the sequence $\sigma$. In particular, in the cyclic case
the transience criterion is $\sum_{k \in \cS}   \cot ( \pi \alpha_k ) > 0$. 

It is easy to see that, if $\cS$ contains
two elements, the \emph{cotangent criterion} \eqref{eq:intro_criterion} 
is equivalent to the previous one for the one-sided oscillating walk ($\alpha_1 + \alpha_2 < 1$).
For more than two half-lines the criterion is non-linear, and it was necessary to extend the model to more
than two lines in order to see the essence of the behaviour.

More generally, we may choose a random routing rule between lines: on departure from half-line $i \in \cS$, the walk jumps
to half-line $j \in \cS$ with probability $p(i,j)$. The deterministic cyclic
routing schedule is a special case in which $p(i, i') =1$ for $i'$ the successor to $i$ in the cycle.
In fact, this set-up generalizes the arbitrary deterministic routing schedule described above, as follows.
Given the schedule sequence $\sigma$ of length $T$, we may convert this to a cyclic schedule
on an extended state-space consisting of $\mu_k$ copies of line $k$, and then reading $\sigma$ as a permutation.
So the deterministic routing
model is a special case of the model with Markov routing, which will be the focus of the rest of the paper.

Our result again will say that \eqref{eq:intro_criterion} is the criterion for transience, where $\mu_k$ is now the
stationary distribution associated to the stochastic matrix $p(i,j)$. Our general model also permits \emph{two-sided}
increments for the walk from some of the lines, which contribute terms involving $\cot ( \pi \alpha_k /2 )$ to the cotangent criterion
\eqref{eq:intro_criterion}.
These two-sided models also generalize previously studied classical models (see e.g.~\cite{kemperman,shepp,rf}).
Again, it is only in our general setting that the essential nature of the cotangent criterion \eqref{eq:intro_criterion} becomes apparent.

Rather than $\RP \times \cS$, one could work on $\ZP \times \cS$ instead, with mass functions replacing
probability densities; the results would be unchanged.

The paper is organized as follows. In Section~\ref{sec:results} we formally define our model and
describe our main results, which as well as a recurrence classification
include results on existence of moments of return times in the recurrent cases. In Section~\ref{sec:oscillating} we explain
how our general model relates to the special case of the oscillating random walk when $\cS$ has two elements, and state 
our results for that model; in this context the recurrence classification results are already known, but the existence
of moments results are new even here, and are in several important cases sharp. 
The present work was also motivated by some problems concerning many-dimensional, partially homogeneous random walks
similar to models studied in~\cite{cp1,cp2}:
we describe this connection in Section~\ref{sec:many-dim}. The main proofs are presented in Sections~\ref{sec:proofs}, \ref{sec:moments},
and~\ref{sec:critical}, the latter dealing with the critical boundary case which is more delicate and requires
additional work. We collect various technical results in Appendix~\ref{sec:appendix}.

\section{Model and results}
\label{sec:results}

Consider  $(X_n, \xi_n ; n \in \ZP)$, a discrete-time, time-homogeneous Markov process 
with state space $\RP \times \cS$, where $\cS$ is a finite non-empty set.
The state space is equipped with the appropriate Borel sets, namely, sets of the form $B \times A$
where $B \in \cB (\RP)$ is a Borel set in $\RP$, and $A \subseteq \cS$.
The process will be described by:
\begin{itemize}
\item an irreducible stochastic matrix labelled
by $\cS$,
$P = ( p(i,j) ; i,j \in \cS )$; and
\item a collection $(w_i ; i \in \cS)$ of probability density functions, so $w_i : \R \to \RP$ is a Borel
function with $\int_\R w_i (y) \ud y =1$.
\end{itemize}
We view $\RP \times \cS$ as
a complex of half-lines $\RP \times \{k\}$, or \emph{branches}, connected at a central origin $\cO := \{0 \} \times \cS$; at time $n$, the coordinate $\xi_n$ describes
which branch the process is on, and $X_n$ describes the distance along that branch at which the process sits. 
We will call $X_n$ a \emph{random walk} on this complex of branches.

To simplify notation, throughout we write $\Pr_{x,i} [ \, \cdot \, ]$ for $\Pr [ \, \cdot \mid (X_0 , \xi_0 ) = (x,i) ]$, the conditional
probability starting from $(x,i) \in \RP \times \cS$; similarly we use $\Exp_{x,i}$ for the corresponding expectation.
The transition kernel of the process is given for $(x,i) \in \RP \times \cS$, for all Borel sets $B \subseteq \RP$ and all $j \in \cS$,  by
\begin{align}
\label{complex-transition}
  \Pr \left[ ( X_{n+1}, \xi_{n+1} ) \in B \times \{ j \} \mid (X_n, \xi_n) = (x , i) \right] &= \Pr_{x,i} \left[ ( X_{1}, \xi_{1} ) \in B \times \{ j \}  \right] \nonumber\\
& {} \hskip -4.7cm {} 
= p (i,j) \int_B w_i (-z-x) \ud z +\1 { i = j } \int_B w_i (z-x) \ud z.  \end{align}
The dynamics of the process represented by \eqref{complex-transition} can be described algorithmically as follows.
Given $(X_n, \xi_n ) = (x,i) \in \RP \times \cS$, generate (independently) a spatial increment $\varphi_{n+1}$
from the distribution given by $w_i$ and a random index $\eta_{n+1} \in \cS$ according to the  
distribution $p(i, \, \cdot \,)$. Then,
\begin{itemize}
\item if $x + \varphi_{n+1} \geq 0$, set $(X_{n+1} , \xi_{n+1}) = (x + \varphi_{n+1} , i)$; or
\item if $x + \varphi_{n+1} < 0$, set $(X_{n+1}, \xi_{n+1} ) = (| x + \varphi_{n+1} | , \eta_{n+1} )$.
\end{itemize}
In words, the walk takes a $w_{\xi_n}$-distributed step. If this step would bring the walk beyond the origin,
it passes through the origin and switches onto branch $\eta_{n+1}$
(or, if $\eta_{n+1}$ happens to be equal to $\xi_n$, it reflects back along the same branch).

The finite 
irreducible stochastic matrix $P$ is associated with a (unique)
positive invariant probability distribution $(\mu_k ; k \in \cS)$ satisfying 
\begin{equation}
\label{mu_system}
 \sum_{j \in \cS} \mu_j p (j,k) - \mu_k =0 , \text{ for all } k \in \cS .\end{equation}
For future reference, we state the following.
\begin{description}
    \item[\namedlabel{ass:basic}{A0}] 
Let $P = ( p(i,j) ; i,j \in \cS )$ be an irreducible stochastic matrix, and let $(\mu_k ; k \in \cS)$
denote the corresponding invariant distribution.
\end{description}

Our interest here is when the $w_i$ are \emph{heavy tailed}.
We allow two classes of distribution for the $w_i$: \emph{one-sided} or \emph{symmetric}.
It is convenient, then, to partition $\cS$ as $\cS = \So \cup \Ss$ where $\So$ and $\Ss$ are disjoint
sets, representing those branches on which the walk takes, respectively, one-sided and symmetric jumps.
The $w_k$ are then described by
a collection of
positive parameters $(\alpha_k ; k \in \cS)$.

For a probability density function $v: \R \to \RP$, an exponent $\alpha \in (0,\infty)$, 
and a constant $c \in (0,\infty)$,
we write $v \in \fD_{\alpha,c}$ to mean
that there exists $c : \RP \to (0,\infty)$ with $\sup_{y} c(y) < \infty$ and $\lim_{y \to \infty} c  (y) = c$ for which
  \begin{equation}
\label{vdef}
 v (y) = \begin{cases} c  (y) y^{-1-\alpha} & \text{ if } y > 0 \\
0 & \text{ if } y \leq 0 . \end{cases}
\end{equation}

If $v \in \fD_{\alpha,c}$ is such that
\eqref{vdef} holds and $c(y)$ satisfies the stronger condition $c(y) = c + O (y^{-\delta})$
for some $\delta>0$, then we write $v \in \fD_{\alpha,c}^+$.

Our assumption on the increment distributions $w_i$ is as follows.
\begin{description}
    \item[\namedlabel{ass:tails}{A1}] 
Suppose that, for each $k \in \cS$,  we have an exponent 
$\alpha_k \in (0,\infty)$, a constant $c_k \in (0,\infty)$, and a density function $v_k \in \fD_{\alpha_k,c_k}$.
Then suppose that, for all $y \in \R$, $w_k$ is given by 
\begin{equation}
\label{wdef}
w_k (y)  = \begin{cases} v_k (-y) & \text{if } k \in \So \\
                  \frac{1}{2} v_k (|y|) & \text{if } k \in \Ss .\end{cases}
\end{equation}
\end{description}

We say that $X_n$ is \emph{recurrent} if $\liminf_{n \to \infty} X_n = 0$, a.s., and \emph{transient} if $\lim_{n \to \infty} X_n = \infty$, a.s.
An irreducibility argument shows that our Markov chain $(X_n,\xi_n)$ displays the usual recurrence/transience dichotomy
and exactly one of these two situations holds; however, our proofs establish this behaviour directly using semimartingale arguments, 
and so we may avoid discussion of irreducibility here.

Throughout we define, for $k \in \cS$,
\[ \chi_k := \frac{ 1 + \1 { k \in \So } }{2} = \begin{cases} 
\frac{1}{2} & \text{ if } k \in \Ss;\\
1  & \text{ if } k \in \So.\end{cases}  \]

Our first main result gives a recurrence classification for the process.

\begin{theorem}
\label{thm:recurrence}
Suppose that \eqref{ass:basic} and \eqref{ass:tails} hold.  
\begin{itemize}
\item[(a)] Suppose that $\max_{k \in \cS} \chi_k \alpha_k \geq 1$. Then $X_n$ is recurrent.
\item[(b)] Suppose that $\max_{k \in \cS} \chi_k \alpha_k < 1$.
\begin{itemize}
\item[(i)] 
If $\sum_{k \in \cS} \mu_k \cot ( \chi_k \pi \alpha_k ) < 0$, then  $X_n$  is recurrent.
\item[(ii)]
If $\sum_{k \in \cS} \mu_k \cot ( \chi_k \pi \alpha_k ) > 0$, then $X_n$  is transient.
\item[(iii)] Suppose  
in addition that the densities $v_{k}$ of \eqref{ass:tails}
belong to $\fD_{\alpha_k,c_k}^+$ for each $k$. Then
$\sum_{k \in \cS} \mu_k \cot ( \chi_k \pi \alpha_k ) = 0$ implies that $X_n$  is recurrent.
\end{itemize}
\end{itemize}
\end{theorem}

In the recurrent cases, it is of interest to quantify recurrence via existence or non-existence
of passage-time moments. For $a > 0$, let $\tau_a := \min \{ n \geq 0 : X_n \leq a \}$, where throughout the paper we adopt
 the usual convention that $\min \emptyset := +\infty$.
The next result shows that in all the recurrent cases, excluding the boundary case in Theorem~\ref{thm:recurrence}(b)(iii),
the tails of $\tau_a$ are \emph{polynomial}.

\begin{theorem}
\label{thm:moments}
Suppose that~\eqref{ass:basic} and~\eqref{ass:tails} hold.  In  cases (a) and (b)(i)
of Theorem~\ref{thm:recurrence}, there exists $0<q_\star<\infty$ such that for all $x>a$ and all  $k\in\cS$, 
\[\Exp_{x,k}[ \tau_a^{q} ] < \infty, \ \textrm{ for } q<q_\star \ \text{ and } \Exp_{x,k}[ \tau_a^{q} ] = \infty, \ \text{ for } q>q_\star.\]
\end{theorem}
\begin{remark}
One would like to precisely locate $q_\star$; here we only locate $q_\star$ within an interval, i.e., we show that there exist $q_0$ and $q_1$ with $0<q_0<q_1<\infty$ such that   $\Exp_{x,k} [ \tau_a^{q} ] < \infty$ for $q<q_0$ and $\Exp_{x,k} [ \tau_a^{q} ] = \infty$ for $q>q_1$. The existence of a critical $q_\star$ follows from the monotonicity of the map $q\mapsto \Exp_{x,k} [ \tau_c^{q} ]$, but obtaining a sharp estimate of its value remains an open problem in the general case.
\end{remark}

We do have sharp results
in several particular cases for \emph{two} half-lines, 
in which case our model reduces to the \emph{oscillating random walk}
considered by Kemperman~\cite{kemperman} and others.
We present these sharp moments results (Theorems~\ref{thm:symmetric_random_walk_moments} and~\ref{thm:one-sided_moments})
in the next section,
which discusses in detail the case of the  oscillating random walk,
and also describes how our recurrence results relate to the known results for this
classical model.

\section{Oscillating random walks and related examples}
\label{sec:oscillating}

\subsection{Two half-lines become one line}

In the case of our general model in which $\cS$ consists of two elements,
$\cS = \{ -1, +1\}$, say, it is natural and convenient
to represent our random walk on the whole real line $\R$. Namely, if 
$\omega (x, k) := k x$
for $x \in \RP$ and $k = \pm1$, we let $Z_n = \omega (X_n, \xi_n)$.

The simplest case has no reflection at the origin, only transmission, i.e.\ $p(i,j) = \1 { i \neq j }$,
so that $\mu = (\frac{1}{2}, \frac{1}{2})$.
Then, for $B \subseteq \R$ a Borel set, 
\begin{align*}
& {} \Pr  [ Z_{n+1} \in B \mid (X_n, \xi_n ) = (x, i) ]  = \\
& \qquad \qquad {}  \Pr_{x,i} [ (X_{1}, \xi_{1} ) \in B^+\! \times \!\{ +1 \}  ]  + \Pr_{x,i} [ (X_{1}, \xi_{1} ) \in B^-\! \times \!\{ -1 \}   ] ,\end{align*}
where $B^+ = B \cap \RP$ and $B^- = \{ - x : x \in B, \, x <0 \}$.
In particular, by \eqref{complex-transition}, writing $w_+$ for $w_{+1}$, for $x \in \RP$,
\begin{align*}  \Pr [ Z_{n+1} \in B \mid (X_n, \xi_n ) = (x, +1) ] &= \int_{B^+} w_{+} (z-x ) \ud z + \int_{B^-} w_{+} (-z-x) \ud z \\
& = \int_B w_{+} (z-x) \ud z ,\end{align*}
and, similarly, writing $w_-(y)$ for $w_{-1}(-y)$, for $x \in \RP$,
  \begin{align*}  \Pr [ Z_{n+1} \in B \mid (X_n, \xi_n ) = (x, -1) ]   = \int_B w_{-} (z+x) \ud z . \end{align*}
For $x \neq 0$,  $\omega$ is invertible
with 
\[ \omega^{-1} (x) = \begin{cases} ( |x| , +1 ) & \text{if } x > 0 \\
( |x| , -1) & \text{if } x < 0 ,\end{cases} \]
and hence we have for $x \in \R \setminus \{ 0 \}$ and Borel $B \subseteq \R$,
  \begin{align}
\label{eq:one-line-kernel}
  \Pr [ Z_{n+1} \in B \mid Z_n = x ]   =  \begin{cases} \int_B w_{+} (z-x) \ud z & \text{if } x > 0 \\
\int_B w_{-} (z-x) \ud z & \text{if } x < 0 .\end{cases}
 \end{align}
We may make an arbitrary non-trivial choice for the transition law at $Z_n = 0$ without affecting the behaviour
of the process, and then~\eqref{eq:one-line-kernel} shows that $Z_n$ is 
  a time-homogeneous Markov process on $\R$.
Now $Z_n$ is \emph{recurrent} if $\liminf_{n \to \infty} | Z_n | = 0$, a.s., or \emph{transient} if $\lim_{n \to \infty} | Z_n | =  \infty$, a.s.
 The one-dimensional case described at~\eqref{eq:one-line-kernel} has received significant attention over the years. We describe
several of the classical models that have been considered.

\subsection{Examples and further results}

\subsubsection*{Homogeneous symmetric random walk}
The most classical case is the following. 
\begin{description}
    \item[\namedlabel{ass:symmetric}{Sym}] 
Let	$\alpha \in (0,\infty)$. For $v \in \fD_{\alpha,c}$, suppose that
$w_+(y) = w_-(y) = \frac{1}{2} v ( |y| )$ for $ y\in \R$. 
\end{description}
In this case, $Z_n$ describes a  random walk with i.i.d.\ symmetric
increments. 
\begin{theorem}  
\label{thm:symmetric_random_walk}
Suppose that  \eqref{ass:symmetric} holds. 
Then the symmetric random walk is transient if   $\alpha < 1$ and recurrent if $\alpha > 1$.
If, in addition, $v  \in \fD^+_{\alpha,c}$, then the case $\alpha =1$ is recurrent.
\end{theorem}
Theorem~\ref{thm:symmetric_random_walk} follows from our Theorem~\ref{thm:recurrence}, since in this case
\[ \sum_{k \in \cS} \mu_k \cot ( \chi_k \pi \alpha_k ) =   \cot( \pi \alpha / 2) .\]
Since it deals with a sum of i.i.d.\ random variables, Theorem~\ref{thm:symmetric_random_walk}
may be deduced from the classical theorem of Chung and Fuchs~\cite{cf}, via e.g.~the formulation of Shepp~\cite{shepp}.
The method of the present paper provides an alternative to the classical (Fourier analytic) approach that generalizes
beyond the i.i.d.\ setting. (Note that Theorem~\ref{thm:symmetric_random_walk} is not, formally, a consequence of Shepp's
most accessible result, Theorem 5 of~\cite{shepp}, since 
$v$ does  not necessarily correspond to a unimodal distribution in Shepp's sense.)

With $\tau_a$ as defined previously, in the setting of the present section we have
 $\tau_a= \min \{ n \geq 0 : | Z_n | \leq a \}$. Use $\Exp_x [ \, \cdot \, ]$
as shorthand for $\Exp [ \, \cdot \, \mid Z_0 = x]$. We have the following result on
existence of passage-time moments, whose proof is in Section~\ref{sec:moments}; while part~(i) is well known,
we could find no reference for part~(ii).

\begin{theorem}  
\label{thm:symmetric_random_walk_moments}
Suppose that~\eqref{ass:symmetric} holds, and that $\alpha > 1$. Let $x  \notin [ -a,  a]$.
\begin{itemize}
\item[(i)] If $\alpha \geq 2$, then  $\Exp_{x} [ \tau_a^q ] < \infty$ if $q < 1/2$ and $\Exp_{x} [ \tau_a^q ] = \infty$ if $q > 1/2$.
\item[(ii)]  If $\alpha \in (1,2)$, then $\Exp_x [ \tau_a^q ] < \infty$ if $q < 1 - \frac{1}{\alpha}$ and
$\Exp_{x} [ \tau_a^q ] = \infty$ if $q > 1 - \frac{1}{\alpha}$.
\end{itemize}
\end{theorem}
Our main interest concerns spatially inhomogeneous models, i.e.,  in which $w_x$ depends on $x$, typically
only through $\sign x$, the sign of $x$. 
Such models are known as \emph{oscillating random walks}, and were
studied by Kemperman~\cite{kemperman}, to whom the model was suggested in 1960 by Anatole Joffe and Peter Ney  (see \cite[p.\ 29]{kemperman}).

\subsubsection*{One-sided oscillating random walk}
The next example,  following~\cite{kemperman}, is a \emph{one-sided oscillating random walk}:
\begin{description}
    \item[\namedlabel{ass:one-sided}{Osc1}] 
Let $\alpha, \beta \in (0,\infty)$. For $v_+ \in \fD_{\alpha,c_+}$ and $v_- \in \fD_{\beta,c_-}$, suppose that
\[ w_+ ( y ) = v_+ (-y) , ~~ \text{and}~~  w_- (y ) = v_- (y) .  \]
\end{description}
In other words, the walk always jumps in the direction of (and possibly over) the origin,
with tail exponent $\alpha$ from the positive half-line and exponent $\beta$ from the negative half-line. 
The following recurrence classification applies.

\begin{theorem}  
\label{thm:one-sided}
Suppose that   \eqref{ass:one-sided} holds. 
Then the one-sided oscillating random walk is transient if  $\alpha + \beta < 1$ and recurrent if $\alpha + \beta > 1$.
If, in addition, $v_+  \in \fD^+_{\alpha,c_+}$ and $v_- \in \fD^+_{\beta,c_-}$, then the case $\alpha + \beta = 1$ is recurrent.
\end{theorem}
Theorem~\ref{thm:one-sided} was obtained in the discrete-space case by Kemperman \cite[p.\ 21]{kemperman};
it follows from our Theorem~\ref{thm:recurrence}, since in this case
\[ \sum_{k \in \cS} \mu_k \cot ( \chi_k \pi \alpha_k ) =  \frac{1}{2} \cot( \pi \alpha ) + \frac{1}{2} \cot (\pi \beta) = 
 \frac{\sin (\pi ( \alpha + \beta) )}{  2 \sin (\pi \alpha ) \sin (\pi \beta ) } .\]

The  special case of  \eqref{ass:one-sided} in which $\alpha = \beta$ was called  \emph{antisymmetric} by Kemperman; here
Theorem \ref{thm:one-sided} shows that the walk is transient for $\alpha < 1/2$ and recurrent for $\alpha > 1/2$. We have the following moments result,
proved in Section~\ref{sec:moments}.

\begin{theorem}  
\label{thm:one-sided_moments}
Suppose that  \eqref{ass:one-sided} holds, and that $\alpha = \beta > 1/2$. Let $x  \notin [ -a,  a]$.
\begin{itemize}
\item[(i)] If $\alpha \geq 1$, then  $\Exp_x [ \tau_a^q ] < \infty$ if  $q < \alpha$ and $\Exp_x [ \tau_a^q ] = \infty$ if  $q > \alpha$.
\item[(ii)] If $\alpha \in (1/2,1)$, then  $\Exp_x [ \tau_a^q ] < \infty$ if  $q < 2 - \frac{1}{\alpha}$ and  $\Exp_x [ \tau_a^q ] = \infty$ if  $q > 2 - \frac{1}{\alpha}$.
\end{itemize}
\end{theorem}

\begin{problem}
Obtain sharp moments results for general $\alpha, \beta \in (0,1)$.
\end{problem}

\subsubsection*{Two-sided oscillating random walk}
Another
 model in the vein of~\cite{kemperman} is 
a \emph{two-sided oscillating random walk}:
\begin{description}
    \item[\namedlabel{ass:two-sided}{Osc2}] 
		Let $\alpha, \beta \in (0,\infty)$. For $v_+ \in \fD_{\alpha,c_+}$ and $v_- \in \fD_{\beta,c_-}$, suppose that
\[ w_+ ( y ) = \frac{1}{2} v_+ (|y|) , ~~ \text{and}~~  w_- (y ) = \frac{1}{2} v_- (|y|) .  \]	
\end{description}
Now the jumps of the walk are symmetric, as under
\eqref{ass:symmetric}, but with a tail exponent
depending upon which side of the origin the walk is currently on, as under \eqref{ass:one-sided}.

The most general recurrence classification result for the model \eqref{ass:two-sided}  is due to Sandri\'c \cite{sandric2}.
A somewhat less general, discrete-space version was obtained by Rogozin and Foss
(Theorem 2 of \cite[p.\ 159]{rf}), building on~\cite{kemperman}. Analogous results
 in continuous time were given in \cite{franke,bottcher}.  Here is the result.

\begin{theorem} 
\label{thm:two-sided}
Suppose that   \eqref{ass:two-sided} holds. 
Then the two-sided oscillating random walk is transient if  $\alpha + \beta < 2$ and recurrent if $\alpha + \beta > 2$.
If, in addition, $v_+  \in \fD^+_{\alpha,c_+}$ and $v_- \in \fD^+_{\beta,c_-}$, then the case $\alpha + \beta = 2$ is recurrent.
\end{theorem}
Theorem~\ref{thm:two-sided} also follows from our Theorem~\ref{thm:recurrence}, since in this case
\[ \sum_{k \in \cS} \mu_k \cot ( \chi_k \pi \alpha_k ) =  \frac{1}{2} \cot( \pi \alpha /2 ) + \frac{1}{2} \cot (\pi \beta /2 ) = 
 \frac{\sin (\pi ( \alpha + \beta) /2 )}{  2 \sin (\pi \alpha /2 ) \sin (\pi \beta /2 ) } .\]

\begin{problem}
Obtain sharp moments results for $\alpha, \beta \in (0,2)$.
\end{problem}

\subsubsection*{Mixed oscillating random walk}
A final model is another oscillating walk that mixes the one- and two-sided models:
\begin{description}
    \item[\namedlabel{ass:mixture}{Osc3}] 
				Let $\alpha, \beta \in (0,\infty)$. For $v_+ \in \fD_{\alpha,c_+}$ and $v_- \in \fD_{\beta,c_-}$, suppose that
\[ w_+ ( y ) = \frac{1}{2} v_+ (|y|) , ~~ \text{and}~~  w_- (y ) =   v_- ( y ) .  \]	
\end{description}

In the discrete-space case, Theorem 2 of Rogozin and Foss \cite[p.\ 159]{rf} gives the recurrence classification.

\begin{theorem} 
\label{thm:mixture}
Suppose that  \eqref{ass:mixture} holds. 
Then the mixed oscillating random walk is transient if  $\alpha + 2 \beta < 2$ and recurrent if $\alpha + 2 \beta > 2$.
If, in addition, $v_+  \in \fD^+_{\alpha,c_+}$ and $v_- \in \fD^+_{\beta,c_-}$, then the case $\alpha + 2 \beta = 2$ is recurrent.
\end{theorem}
Theorem~\ref{thm:mixture} also follows from our Theorem~\ref{thm:recurrence}, since in this case
\[ \sum_{k \in \cS} \mu_k \cot ( \chi_k \pi \alpha_k ) =  \frac{1}{2} \cot( \pi \alpha /2 ) + \frac{1}{2} \cot (\pi \beta   ) = 
 \frac{\sin (\pi ( \alpha + 2 \beta) /2 )}{  2 \sin (\pi \alpha /2 ) \sin (\pi \beta   ) } .\]

\subsection{Additional remarks}

It is possible to generalize the model further by permitting the local transition density to vary \emph{within} each half-line.
Then we have the transition kernel
\begin{equation}
\label{markov}
 \Pr [ Z_{n+1} \in B \mid Z_n = x ] = \int_B w_x (z -x ) \ud z,  \end{equation}
for all Borel sets $B \subseteq \R$. Here the local transition densities $w_x : \R \to \RP$ are 
Borel functions.
Variations of the oscillating random walk, within the general setting of \eqref{markov}, have also been studied in the literature. 
Sandri\'c \cite{sandric1,sandric2} supposes that the $w_x$ satisfy, 
for each $x \in \R$, $w_x (y) \sim c(x) | y |^{-1-\alpha (x)}$
as $|y| \to \infty$ for some measurable functions $c$ and $\alpha$;
he refers to this as a \emph{stable-like Markov chain}. Under a uniformity condition on the $w_x$,
and other mild technical conditions, Sandri\'c \cite{sandric1} obtained, via Foster--Lyapunov methods similar in spirit to those of the present paper,
 sufficient conditions
for recurrence and transience: essentially $\liminf_{x \to \infty} \alpha (x) >1$ is sufficient for recurrence and
$\limsup_{x \to \infty} \alpha (x) <1$ is sufficient for transience.
These results can be seen as a generalization of Theorem \ref{thm:symmetric_random_walk}. 
Some related results for models in continuous-time (L\'evy processes)
are given in \cite{sw,wang,sandric3}. 
Further results and an overview of the literature are provided in Sandri\'c's PhD thesis \cite{sandric0}.

\section{Many-dimensional random walks}
\label{sec:many-dim}

The next two examples show how versions of the oscillating random walk of Section~\ref{sec:oscillating} arise as embedded Markov chains in certain two-dimensional random walks.

\begin{example}
\label{ex:cp1}
Consider $\xi_n =(\xi_n^{(1)}, \xi_n^{(2)})$, $n \in \ZP$,
 a nearest-neighbour random walk on $\Z^2$
with transition probabilities\index{random walk!multidimensional}\index{random walk!simple}
\[ \Pr [ \xi_{n+1} = (y_1, y_2 ) \mid \xi_n = (x_1, x_2 ) ] = p (x_1, x_2 ; y_1, y_2 ) . \]
Suppose that the probabilities are given for $x_2 \neq 0$ by,
\begin{align}
\label{e:cp_jumps}
 p (x_1, x_2 ; x_1 , x_2 + 1 ) & =  p (x_1, x_2 ; x_1 , x_2 - 1 ) = \frac{1}{3} ; \nonumber\\
 p (x_1, x_2 ; x_1 + 1, x_2   ) & = \frac{1}{3}  \1 { x_2 < 0  }; \nonumber\\
 p (x_1, x_2 ; x_1 - 1, x_2   ) & = \frac{1}{3}  \1 { x_2 > 0  };  
\end{align}
(the rest being zero)
and for $x_2 = 0$ by $p (x_1, 0 ; x_1 , 1 ) = 1$ for all $x_1 > 0$, 
$p (x_1, 0 ; x_1 , -1 ) = 1$ for all $x_1 < 0$,
and $p (0, 0 ; 0 , 1 ) = p (0, 0 ; 0 , -1 ) = 1/2$. See   Figure~\ref{fig:camp_pet2} for an illustration.
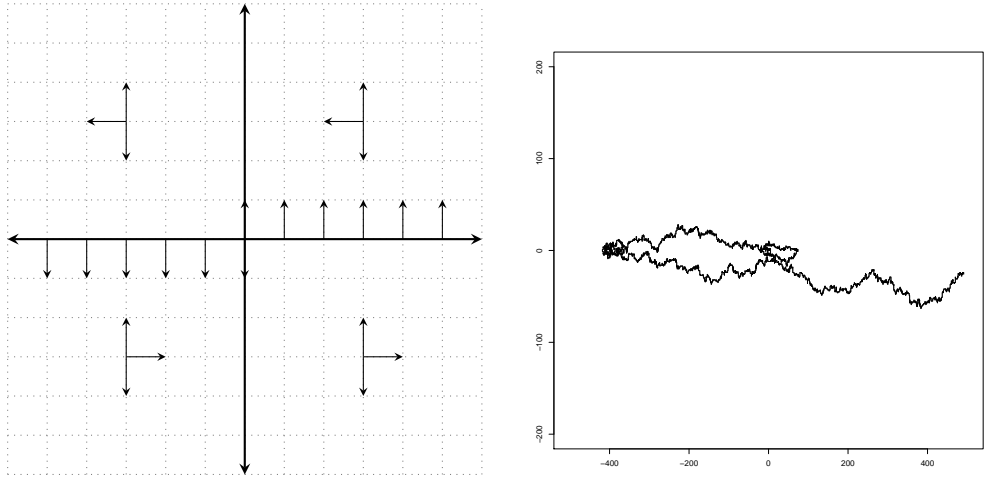
\begin{figure}[!h]
\begin{center}
\begin{tikzpicture}[scale=1.3]
\foreach \i in {0,...,12} {
        \draw [dotted,gray] (0.4*\i,0) -- (0.4*\i,0.4*12);
                                   }
\foreach \i in {0,...,12} {
        \draw [dotted,gray] (0,0.4*\i) -- (0.4*12,0.4*\i);
                                   }
\draw [<->,>=stealth,thick] (0*6,0.4*6) -- (0.4*12,0.4*6); 
\draw [<->,>=stealth,thick] (0.4*6,0) -- (0.4*6,0.4*12); 	
\draw [ ->,>=stealth] (0.4*9,0.4*9) -- (0.4*9,0.4*8);
\draw [ ->,>=stealth] (0.4*9,0.4*9) -- (0.4*9,0.4*10);
\draw [ ->,>=stealth] (0.4*9,0.4*9) -- (0.4*8,0.4*9);
\draw [ ->,>=stealth] (0.4*3,0.4*9) -- (0.4*3,0.4*8);
\draw [ ->,>=stealth] (0.4*3,0.4*9) -- (0.4*3,0.4*10);
\draw [ ->,>=stealth] (0.4*3,0.4*9) -- (0.4*2,0.4*9);
\foreach \i in {1,...,5} {
	\draw [  ->,>=stealth] (0.4*\i,0.4*6) -- (0.4*\i,0.4*5);
			        }
\foreach \i in {7,...,11} {
	\draw [  ->,>=stealth] (0.4*\i,0.4*6) -- (0.4*\i,0.4*7);
			         }
\draw [ <->,>=stealth] (0.4*6,0.4*5) -- (0.4*6,0.4*7);
\draw [ ->,>=stealth] (0.4*9,0.4*3) -- (0.4*9,0.4*2);
\draw [ ->,>=stealth] (0.4*9,0.4*3) -- (0.4*9,0.4*4);
\draw [ ->,>=stealth] (0.4*9,0.4*3) -- (0.4*10,0.4*3);
\draw [ ->,>=stealth] (0.4*3,0.4*3) -- (0.4*3,0.4*2);
\draw [ ->,>=stealth] (0.4*3,0.4*3) -- (0.4*3,0.4*4);
\draw [ ->,>=stealth] (0.4*3,0.4*3) -- (0.4*4,0.4*3);
\end{tikzpicture}
\hskip5mm
 \includegraphics[width=0.4\textwidth]{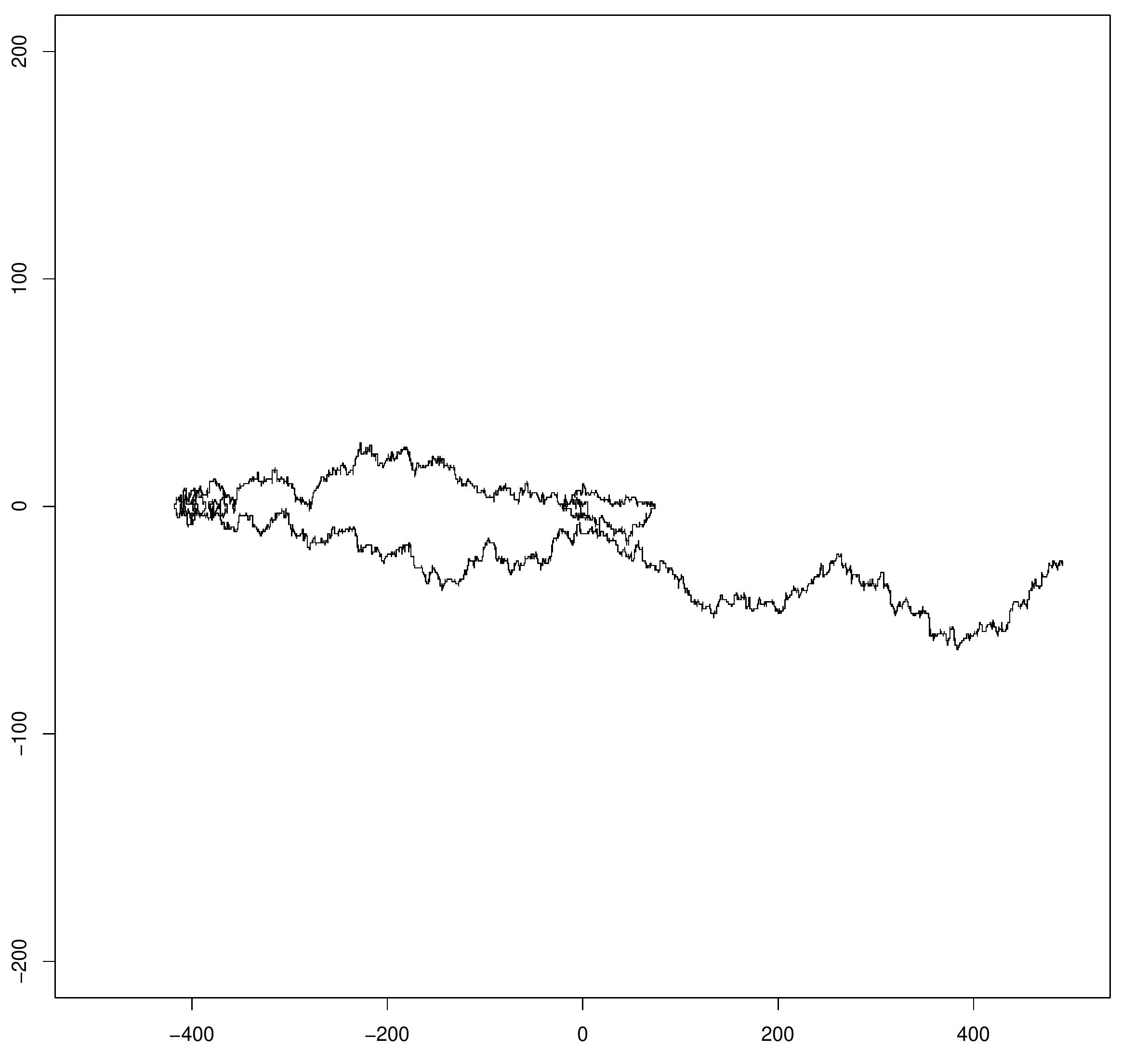}
 \caption{\label{fig:camp_pet2} Pictorial representation  of the non-homogeneous nearest-neighbour random walk on $\Z^2$
of Example~\ref{ex:cp2},
plus a simulated trajectory of 5000 steps of the walk.
We conjecture that the walk
is recurrent. }
 \end{center}
 \end{figure}

Set $\tau_0 :=0$ and define recursively $\tau_{k+1} = \min \{ n > \tau_k : \xi^{(2)}_n = 0 \}$
for $k \geq 0$; consider the embedded Markov chain
$X_n = \xi^{(1)}_{\tau_n}$. We show that $X_n$ is a discrete version of the oscillating random walk
described in Section~\ref{sec:oscillating}.
Indeed, $| \xi^{(2)}_n |$ is a reflecting random walk on $\ZP$ with increments taking values $-1, 0, +1$ each with probability
$1/3$. We then (see e.g.~\cite[p.~415]{feller2}) have that for some constant $c \in (0,\infty)$, 
\[ \Pr [ \tau_1 > r ] = (c + o(1) ) r^{-1/2}, ~\text{as}~ r \to \infty.\]
 Suppose that $\xi^{(1)}_0 = x >0$. 
Since between times $\tau_0$ and $\tau_1$, $\xi^{(1)}_n$ is monotone, we have
\[ \Pr [ \xi_{\tau_1}^{(1)} - \xi_{\tau_0}^{(1)} < -r ] \geq  \Pr [ \tau_1 > 3r + r^{3/4} ] - \Pr [ \xi^{(1)}_{3r + r^{3/4}} - \xi^{(1)}_0 \geq  -r ]  .\]
Here, by the Azuma--Hoeffding inequality, for some $\eps >0$ and all $r \geq 1$,
\[ \Pr [ \xi^{(1)}_{3r + r^{3/4}} - \xi^{(1)}_0  \geq  -r ] \leq \exp \{ - \eps r^{1/2} \} . \]
Similarly,
\[\Pr [ \xi_{\tau_1}^{(1)} - \xi_{\tau_0}^{(1)} < -r ] \leq   \Pr [ \tau_1 > 3r - r^{3/4} ] + \Pr [ \xi^{(1)}_{3r-  r^{3/4}} - \xi^{(1)}_0  \leq  -r ]   ,\]
where $\Pr [ \xi^{(1)}_{3r-  r^{3/4}} - \xi^{(1)}_0 \leq  -r ] \leq \exp \{ - \eps r^{1/2} \}$.
Combining these bounds, and using the symmetric argument for $\{ \xi_{\tau_1}^{(1)} > r\}$ when 
$\xi^{(1)}_0 = x <0$, 
we see that for $r > 0$, 
\begin{align}
\label{e:cp1_jump}
\Pr [ X_{n+1} - X_n < -r \mid X_n = x ] & = u(r), \text{ if } x > 0, \text{ and } \nonumber\\
\Pr [ X_{n+1} - X_n > r \mid X_n = x ] & = u(r), \text{ if } x < 0 , 
\end{align}
where $u(r)  =  ( c + o(1) ) r^{-1/2}$. 
Thus $X_n$ satisfies a discrete-space analogue of~\eqref{ass:one-sided} 
with $\alpha = \beta = 1/2$. This is the critical case identified in Theorem~\ref{thm:one-sided},
but that result does not cover this case due to the rate of convergence estimate for $u$;
 a finer analysis is required. We conjecture that the walk is recurrent.
\end{example}

\begin{example}
\label{ex:cp2}
We present two variations on the previous example, which are superficially similar but turn out to be less delicate.
First, modify the random walk of the previous example
by supposing that~\eqref{e:cp_jumps} holds
but replacing the behaviour at $x_2 = 0$ by
$p (x_1, 0 ; x_1 , 1 ) = p (x_1, 0 ; x_1, -1 ) = 1/2$
for all $x_1 \in \Z$. See the left-hand part of Figure~\ref{fig:camp_pet1} for an illustration.

The embedded process $X_n$ now has, for all $x \in \Z$ and for $r \geq 0$,
\begin{equation}
\label{e:cp2_jump}
 \Pr [ X_{n+1} - X_n < -r \mid X_n = x ] =  \Pr [ X_{n+1} - X_n > r \mid X_n = x ] = u(r) ,  
\end{equation}
where $u(r) =  (c/2) ( 1 + o(1) ) r^{-1/2}$. 
Thus $X_n$ is a random walk with symmetric increments, and the discrete version of our Theorem~\ref{thm:symmetric_random_walk} (and also a result of~\cite{shepp})
implies that the walk is transient. 
This walk was studied by Campanino and Petritis \cite{cp1,cp2}, who proved  transience via different methods. 
\begin{figure}[!h]
\begin{center}
\scalebox{1.3}{
\begin{tikzpicture}
    \foreach \i in {0,...,12} {
        \draw [dotted,gray] (0.4*\i,0) -- (0.4*\i,0.4*12);
    }
    \foreach \i in {0,...,12} {
        \draw [dotted,gray] (0,0.4*\i) -- (0.4*12,0.4*\i);
    }
		
	  \draw [<->,>=stealth,thick] (0*6,0.4*6) -- (0.4*12,0.4*6); 
    \draw [<->,>=stealth,thick] (0.4*6,0) -- (0.4*6,0.4*12); 
		
		 	\draw [ ->,>=stealth] (0.4*9,0.4*9) -- (0.4*9,0.4*8);
			\draw [ ->,>=stealth] (0.4*9,0.4*9) -- (0.4*9,0.4*10);
			\draw [ ->,>=stealth] (0.4*9,0.4*9) -- (0.4*8,0.4*9);
			
			 	\draw [ ->,>=stealth] (0.4*3,0.4*9) -- (0.4*3,0.4*8);
			\draw [ ->,>=stealth] (0.4*3,0.4*9) -- (0.4*3,0.4*10);
			\draw [ ->,>=stealth] (0.4*3,0.4*9) -- (0.4*2,0.4*9);
			
			\foreach \i in {1,...,11} {
			\draw [ <->,>=stealth] (0.4*\i,0.4*5) -- (0.4*\i,0.4*7);
			}
			
	  	\draw [ ->,>=stealth] (0.4*9,0.4*3) -- (0.4*9,0.4*2);
			\draw [ ->,>=stealth] (0.4*9,0.4*3) -- (0.4*9,0.4*4);
			\draw [ ->,>=stealth] (0.4*9,0.4*3) -- (0.4*10,0.4*3);
			
			 	\draw [ ->,>=stealth] (0.4*3,0.4*3) -- (0.4*3,0.4*2);
			\draw [ ->,>=stealth] (0.4*3,0.4*3) -- (0.4*3,0.4*4);
			\draw [ ->,>=stealth] (0.4*3,0.4*3) -- (0.4*4,0.4*3);
			
		\end{tikzpicture}
		\qquad
\begin{tikzpicture}
    \foreach \i in {0,...,12} {
        \draw [dotted,gray] (0.4*\i,0) -- (0.4*\i,0.4*12);
    }
    \foreach \i in {0,...,12} {
        \draw [dotted,gray] (0,0.4*\i) -- (0.4*12,0.4*\i);
    }
		
	  \draw [<->,>=stealth,thick] (0*6,0.4*6) -- (0.4*12,0.4*6); 
    \draw [<->,>=stealth,thick] (0.4*6,0) -- (0.4*6,0.4*12); 
				
	  	\draw [ ->,>=stealth] (0.4*9,0.4*9) -- (0.4*9,0.4*8);
			\draw [ ->,>=stealth] (0.4*9,0.4*9) -- (0.4*9,0.4*10);
			\draw [ ->,>=stealth] (0.4*9,0.4*9) -- (0.4*8,0.4*9);
			
			 	\draw [ ->,>=stealth] (0.4*3,0.4*9) -- (0.4*3,0.4*8);
			\draw [ ->,>=stealth] (0.4*3,0.4*9) -- (0.4*3,0.4*10);
			\draw [ ->,>=stealth] (0.4*3,0.4*9) -- (0.4*2,0.4*9);
			
			\foreach \i in {1,...,5} {
			\draw [ ->,>=stealth] (0.4*\i,0.4*6) -- (0.4*\i,0.4*5);
			}
			
			\foreach \i in {6,...,11} {
			\draw [  <->,>=stealth] (0.4*\i,0.4*5) -- (0.4*\i,0.4*7);
			}
			
	  	\draw [ ->,>=stealth] (0.4*9,0.4*3) -- (0.4*9,0.4*2);
			\draw [ ->,>=stealth] (0.4*9,0.4*3) -- (0.4*9,0.4*4);
			\draw [ ->,>=stealth] (0.4*9,0.4*3) -- (0.4*10,0.4*3);
			
			 	\draw [ ->,>=stealth] (0.4*3,0.4*3) -- (0.4*3,0.4*2);
			\draw [ ->,>=stealth] (0.4*3,0.4*3) -- (0.4*3,0.4*4);
			\draw [ ->,>=stealth] (0.4*3,0.4*3) -- (0.4*4,0.4*3);
			
		\end{tikzpicture}		
		}
 \caption{Pictorial representation of the two non-homogeneous nearest-neighbour random walks on $\Z^2$ of
Example~\ref{ex:cp2}. Each of these walks is transient.}
 \label{fig:camp_pet1}
 \end{center}
 \end{figure}
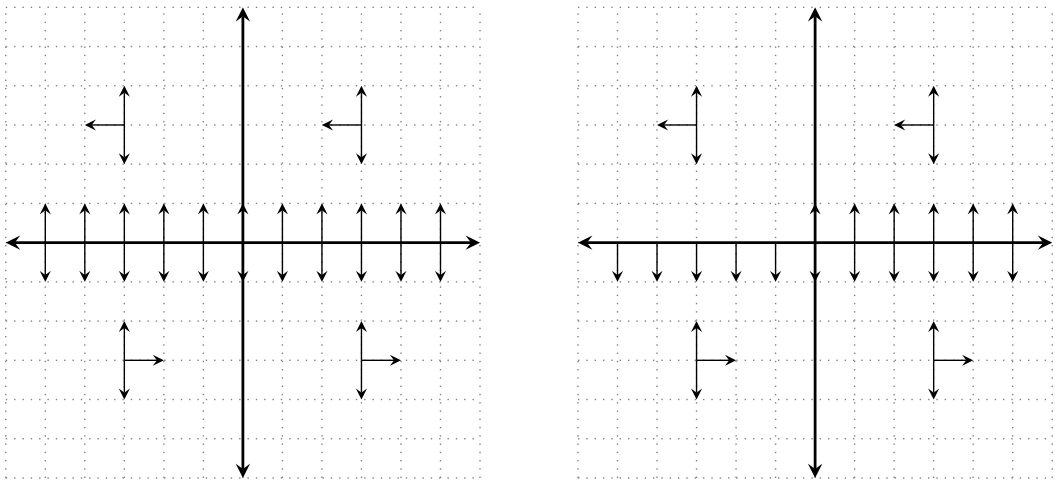

Next, modify the random walk of Example~\ref{ex:cp1}
by supposing that~\eqref{e:cp_jumps} holds
but replacing the behaviour at $x_2 = 0$ by
$p (x_1, 0 ; x_1 , 1 ) = p (x_1, 0 ; x_1, -1 ) = 1/2$ if $x_1 \geq 0$,
and $p(x_1, 0 ; x_1 , -1) = 1$ for $x_1 < 0$. See the right-hand part of Figure~\ref{fig:camp_pet1} for an illustration.
This time the walk takes a symmetric increment as at~\eqref{e:cp2_jump} when $x \geq 0$
but a one-sided increment as at~\eqref{e:cp1_jump} when $x < 0$. In this case the discrete version of our Theorem~\ref{thm:mixture} (and also a result of~\cite{rf}) shows that the walk is transient.
\end{example}

One may obtain the general model on $\ZP \times \cS$ as an embedded process for a random walk on
complexes of half-spaces, generalizing the examples considered here; we leave this to the interested reader.

\section{Recurrence classification in the non-critical cases}
\label{sec:proofs}

\subsection{Lyapunov functions}
\label{sec:lyapunov_functions}

Our proofs are based on demonstrating appropriate Lyapunov functions;
that is, for suitable $\varphi: \RP \times \cS \to \RP$ we study
$Y_n = \varphi ( X_n , \xi_n )$ such that $Y_n$ has
appropriate local supermartingale or submartingale properties for the one-step
mean increments
\begin{align*}
D\varphi(x,i) & := \Exp \left[ \varphi(X_{n+1}, \xi_{n+1} ) - \varphi (X_n, \xi_n ) \mid (X_n, \xi_n ) = (x, i) \right] 
\\
& = \Exp_{x, i} \left[ \varphi(X_{1}, \xi_{1} ) - \varphi (X_0, \xi_0 ) \right] . \end{align*}

First we note some consequences of the transition law \eqref{complex-transition}.
Let $\varphi: \RP \times \cS \to \RP$ be measurable. Then, we have from \eqref{complex-transition} that, for $(x, i) \in \RP \times \cS$,
\begin{align}
\label{general-function}
D\varphi(x,i)
&= \sum_{j \in \cS} p (i,j) \int_{-\infty}^{-x} \left( \varphi( -x-y , j) - \varphi(x,i) \right) w_i (y) \ud y
\nonumber\\
& {} \hskip 3cm {} +  \int_{-x}^\infty \left( \varphi(x+y, i) - \varphi(x,i) \right) w_i (y) \ud y.
\end{align}

Our primary Lyapunov function  is roughly of the form $x \mapsto | x| ^\nu$, $\nu \in \R$,
but weighted according to an $\cS$-dependent
component (realised
by a collection of multiplicative weights $\lambda_k$); these weights provide  a crucial technical tool.

For $\nu \in \R$ and $x \in \R$ we write
\[ f_\nu (x) := (1 + |x| )^\nu .\]
Then, for parameters $\lambda_k > 0$ for each $k \in \cS$, define
for $x \in \RP$ and $k \in \cS$,
\begin{equation}
\label{fdef}
 f_\nu (x , k ) := \lambda_k f_\nu (x) = \lambda_k (1+x)^\nu.
\end{equation}
Now for this Lyapunov function, \eqref{general-function} gives
\begin{align}
\label{multiplicative-function}
Df_\nu (x, i) &=
\sum_{j \in \cS} p (i,j)  \int_{-\infty}^{-x} \left( \lambda_j f_\nu (x+y) - \lambda_i  f_\nu (x) \right) w_i (y) \ud y
 \nonumber\\
& {} \hskip 2cm {}
+ \lambda_i \int_{-x}^\infty \left( f_\nu ( x+ y) - f_\nu (x) \right) w_i (y) \ud y.\end{align}
Depending on whether $i \in \Ss$ or $i \in \So$, the above integrals can be expressed in terms of $v_i$ as follows.
For $i\in\Ss$, 
\begin{align}
\label{conditional-increment-sym}
Df_\nu (x, i) &=\sum_{j \in \cS} p (i,j)  \frac{\lambda_j}{2}  \int_{x}^\infty  
f_\nu (y-x) v_i (y) \ud y  - \frac{\lambda_i}{2} \int_x^\infty  f_\nu (x) v_i (y) \ud y
\nonumber\\
& {} \hskip 0.7cm {}+ \frac{\lambda_i}{2}  \int_{0}^x 
\left( f_\nu ( x+ y) + f_\nu (x-y) - 2f_\nu (x) \right) v_i ( y ) \ud y  \nonumber\\
& {} \hskip 0.7cm {}+ \frac{\lambda_i}{2}   \int_x^\infty \left( f_\nu ( x+ y) - f_\nu (x) \right) v_i (y) \ud y .
\end{align}
For $i\in\So$, 
\begin{align}
 \label{conditional-increment-one}
Df_\nu (x, i) & =\sum_{j \in \cS} p (i,j) \lambda_j \int_x^\infty   f_\nu (y-x) v_i (y) \ud y 
 - \lambda_i \int_x^\infty f_\nu (x)  v_i (y) \ud y \nonumber\\
& {} \hskip 0.7cm {}
+\lambda_i
 \int_{0}^x \left( f_\nu ( x - y) - f_\nu (x) \right) v_i (y) \ud y.
 \end{align}

\subsection{Estimates of functional increments}
\label{sec:increments}

In the course of our proofs, we need various integral estimates 
that can be expressed in terms of classical transcendental functions. 
For the convenience of the reader, we gather all necessary integrals in Lemmas~\ref{lem:i-integrals} and~\ref{lem:j-integrals};
the proofs of these results are deferred until Section~\ref{sec:integrals}.  
Recall that the Euler gamma function $\Gamma$ satisfies the functional equation $z \Gamma (z) = \Gamma (z+1)$,
and the hypergeometric function ${}_m F_n$ is defined via a power series (see~\cite{as}).

\begin{lemma}
\label{lem:i-integrals}
Suppose that $\alpha > 0$ and $-1<\nu<\alpha$. Then
\begin{align*}
i_0^{\nu,\alpha} &:=\int_1^\infty \frac {(1+u)^\nu-1}{u^{1+\alpha}} \ud u= 
 \frac{1}{\alpha-\nu} \ \HG ( -\nu, \alpha-\nu ; \alpha-\nu+1 ; -1) -\frac{1}{\alpha};\\
i_{2,0}^\alpha & :=\int_1^\infty \frac{1}{u^{1+\alpha}} \ud u= \frac{1}{\alpha}; \\
i_{2,1}^{\nu,\alpha} &:= \int_1^\infty \frac {(u-1)^\nu}{u^{1+\alpha}} \ud u = \frac{\Gamma(1+\nu)\Gamma(\alpha-\nu)}{\Gamma(1+\alpha)}.
\end{align*}
Suppose that $\alpha \in (0,2)$ and $\nu > -1$. Then
\[ i_1^{\nu,\alpha}  := \int_0^1 \frac {(1+u)^\nu+(1-u)^\nu-2 }{u^{1+\alpha}} \ud u = \frac{\nu(\nu-1)}{2-\alpha}
\ \HGG (1,1-\tfrac{\nu}{2},1-\tfrac{\alpha}{2}, \tfrac{3-\nu}{2} ; \tfrac{3}{2},2,  2-\tfrac{\alpha}{2} ;  1) . \]
Suppose that $\alpha \in (0,1)$ and $\nu > -1$. Then
\[\tilde{i}_1^{\nu,\alpha}:=  \int_0^1 \frac {(1-u)^\nu-1 }{u^{1+\alpha}} \ud u= \frac{1}{\alpha}
\left( 1 - \frac{\Gamma(1+\nu)\Gamma(1-\alpha)}{\Gamma(1-\alpha+\nu)}\right).\]
\end{lemma}

Recall that the digamma function is $\psi(z) = \frac{\ud}{\ud z} \log \Gamma (z) = \Gamma' (z) / \Gamma (z)$, 
which has $\psi (1) = - \gamma$ where $\gamma \approx 0.5772$ is Euler's constant. 

\begin{lemma}
\label{lem:j-integrals}
Suppose that $\alpha >0$. Then
\begin{align*}
j_0^\alpha &:=\int_1^\infty \frac{\log (1+u)}{u^{1+\alpha}} \ud u= \frac{1}{\alpha}\left(\psi(\alpha)-\psi(\tfrac{\alpha}{2}) \right); \\
j_2^\alpha &:=\int_1^\infty \frac{\log (u-1)}{u^{1+\alpha}} \ud u= -\frac{1}{\alpha}(\gamma +\psi(\alpha)). 
\end{align*}
Suppose that $\alpha \in (0,2)$. Then
\[ j_1^\alpha :=\int_0^1 \frac{\log (1-u^2)}{u^{1+\alpha}} \ud u= \frac{1}{\alpha}\left(\gamma +\psi(1-\tfrac{\alpha}{2})\right) 
 .\]
Suppose that $\alpha \in (0,1)$. Then
\[ \tilde{j}_1^\alpha :=\int_0^1 \frac{\log (1-u)}{u^{1+\alpha}} \ud u= \frac{1}{\alpha}\left(\gamma +\psi(1-\alpha)\right)
. \]
\end{lemma}

\begin{remark}
The $j$ integrals can be obtained as derivatives with respect to $\nu$ of the  $i$ integrals, evaluated at $\nu=0$.
\end{remark}

The next result collects estimates for our integrals in the expected functional increments~\eqref{conditional-increment-sym} and~\eqref{conditional-increment-one} in terms of the integrals in Lemma~\ref{lem:i-integrals}.

\begin{lemma}
\label{lem:all-integrals}
Suppose that $v \in \fD_{\alpha,c}$.
For $\alpha >0$ and $-1 < \nu < \alpha$ we have
\begin{align}
\label{int21}
\int_x^\infty f_\nu ( y -x ) v (y) \ud y & = c x^{\nu - \alpha} i^{\nu, \alpha}_{2,1} + o (x^{\nu - \alpha} ) ;\\
\label{int20}
\int_x^\infty f_\nu (x) v  (y) \ud y & = c x^{\nu -\alpha} i_{2,0}^{\nu,\alpha} + o (x^{\nu - \alpha} ) ;\\
\label{int0}
\int_x^\infty \left( f_\nu (x+y) - f_\nu (x) \right) v (y) \ud y & = c x^{\nu -\alpha} i_{0}^{\nu,\alpha}+ o (x^{\nu - \alpha} ) .
\end{align}
For $\alpha \in (0,2)$ and $\nu > -1$ we have
\begin{equation}
\label{int1}
 \int_{0}^x \left( f_\nu (x+y) + f_\nu (x-y) - 2 f_\nu (x) \right) v ( y ) \ud y = c x^{\nu - \alpha} i_1^{\nu,\alpha} + o (x^{\nu - \alpha} ) . 
\end{equation}
For $\alpha \in (0,1)$ and $\nu > -1$ we have
\begin{equation}
\label{intt1}
 \int_{0}^x \left( f_\nu (x-y) - f_\nu (x) \right) v (y) \ud y = c x^{\nu - \alpha} \tilde i_1^{\nu,\alpha} + o (x^{\nu - \alpha} ) . \end{equation}
Moreover, if $v \in \fD_{\alpha,c}^+$ then stronger versions of  all of the above estimates 
hold with $ o (x^{\nu - \alpha} ) $ replaced by $O (x^{\nu - \alpha-\delta} ) $
for some $\delta >0$. 
\end{lemma}
\begin{proof}
These estimates are mostly quite straightforward, so we do not give all the details. We spell out the estimate
in~\eqref{int21}; the others are similar. We have
\begin{align*} \int_x^\infty f_\nu ( y -x ) v (y) \ud y & = x^\nu \int_x^\infty \left( \tfrac{1+ y}{x} - 1 \right)^\nu c(y) y^{-1-\alpha} \ud y .
\end{align*}
With the substitution $u = \frac{1+y}{x}$, this last expression becomes
\[ x^{1+\nu} \int_{\frac{1+x}{x}}^\infty ( u -1 )^\nu u^{-1-\alpha} (x - u^{-1} )^{-1-\alpha} c ( u x -1 ) \ud u .\]
Let $\eps \in (0,c)$. Then there exists $y_0 \in \RP$ such that $| c(y) - c | < \eps$
for all $y \geq y_0$, so that $| c ( ux -1 ) - c | < \eps$ for all $u$ in the range of integration,
provided $x \geq y_0$. Writing
\[ f(u) = (u-1)^\nu u^{-1-\alpha}, \text{ and } g(u) = (x - u^{-1} )^{-1-\alpha} c ( ux -1 ) ,\]
for the duration of the proof, we have that
\begin{equation}
\label{eq33}
 \int_x^\infty f_\nu ( y -x ) v (y) \ud y = x^{1+\nu} \int_{\frac{1+x}{x}}^\infty f(u) g(u) \ud u .
\end{equation}
For $u \geq \frac{1+x}{x}$ and $x \geq y_0$, we have
\[ g_- := (c-\eps) x^{-1-\alpha} \leq g(u) \leq (c+\eps) (x-1)^{-1-\alpha} =: g_+ ,\]
so that $g_+ - g_- \leq 2 \eps (x-1)^{-1-\alpha} + C_1 x^{-2-\alpha}$
for a constant $C_1 < \infty$ not depending on $x \geq y_0$ or $\eps$. Moreover, it is easy to see that
$\int_1^\infty | f(u) | \ud u \leq C_2$ for a constant $C_2$ depending only on $\nu$
and $\alpha$, provided $\nu \in (-1,\alpha)$. Hence Lemma~\ref{lem:integral_approx} shows that
\[ \left| \int_{\frac{1+x}{x}}^\infty f(u) g(u) \ud u - (c-\eps) x^{-1-\alpha} \int_{\frac{1+x}{x}}^\infty f(u) \ud u
\right| \leq 2 C_2 \eps (x-1)^{-1-\alpha} + C_1 C_2 x^{-2-\alpha} ,\]
for all $x \geq y_0$. Since also
\[ \left| \int_{\frac{1+x}{x}}^\infty f(u) \ud u - i^{\nu,\alpha}_{2,1} \right| \leq \int_1^{\frac{1+x}{x}} | f(u) | \ud u \to 0 ,\]
as $x \to \infty$, it follows that for any $\eps >0$ we may choose $x$ sufficiently large so that
\[ \left| \int_{\frac{1+x}{x}}^\infty f(u) g(u) \ud u -  c x^{-1-\alpha} i^{\nu,\alpha}_{2,1}
\right| \leq \eps x^{-1-\alpha} ,\]
and since $\eps>0$ was arbitrary, we obtain~\eqref{int21} from~\eqref{eq33}.
\end{proof}

We also need the following simple estimates for ranges of $\alpha$ when the asymptotics for the final two
integrals in Lemma~\ref{lem:all-integrals} are not valid.

\begin{lemma}
\label{lem:big-alpha}
Suppose that $v \in \fD_{\alpha,c}$.
\begin{itemize}
\item[(i)]
For $\alpha \geq 2$ and any $\nu \in (0,1)$, there exist $\eps>0$ and $x_0 \in \RP$ such that,
for all $x \geq x_0$,
\[ \int_{0}^x \left( f_\nu (x+y) + f_\nu (x-y) - 2 f_\nu (x) \right) v ( y ) \ud y 
\leq \begin{cases} - \eps x^{\nu -2} \log x & \text{ if } \alpha =2, \\
- \eps x^{\nu - 2} & \text{ if } \alpha > 2. 
\end{cases} \]
\item[(ii)]
For $\alpha \geq 1$ and any $\nu >0$, there exist $\eps>0$ and $x_0 \in \RP$ such that,
for all $x \geq x_0$,
\[ \int_{0}^x \left( f_\nu (x-y) - f_\nu (x) \right) v (y) \ud y 
\leq \begin{cases} - \eps x^{\nu -1} \log x & \text{ if } \alpha =1, \\
- \eps x^{\nu - 1} & \text{ if } \alpha > 1. 
\end{cases} \]
\end{itemize}
\end{lemma}
\begin{proof} 

For part (i), set $a_\nu (z) = (1 + z)^\nu + (1-z)^\nu - 2$, so that
\begin{align*}
 \int_{0}^x \left( f_\nu (x+y) + f_\nu (x-y) - 2 f_\nu (x) \right) v ( y ) \ud y 
& = (1+x)^\nu \int_0^x a_\nu \left( \tfrac{y}{1+x} \right) c( y) y^{-1-\alpha} .\end{align*}
Suppose that $\alpha \geq 2$ and $\nu \in (0,1)$.
For $\nu \in (0,1)$, calculus shows that $a_\nu(z)$ has a single local maximum at $z=0$,
so that $a_\nu(z) \leq 0$ for all $z$. Moreover, Taylor's theorem shows that 
for any $\nu \in (0,1)$ there exists $\delta_\nu \in (0,1)$ such that
$a_\nu (z) \leq - (\nu /2) (1-\nu) z^2$ for all $z \in [0, \delta_\nu]$. 
Also, $c(y) \geq c/2 > 0$ for all $y \geq y_0$ sufficiently large.
Hence, for all $x \geq y_0 / \delta_\nu$,
\begin{align*}
 \int_{0}^x  a_\nu  \left(   \tfrac{y}{1+x} \right) c (y) y^{-1-\alpha} \ud y 
& \leq \int_{y_0}^{\delta_\nu x} a_\nu \left(   \tfrac{y}{1+x} \right) c (y) y^{-1-\alpha} \ud y \\
& \leq - \frac{c \nu (1-\nu)}{4(1+x)^2} \int_{y_0}^{\delta_\nu x} y^{1-\alpha} \ud y , \end{align*}
which yields part  (i) of the lemma.

For part (ii), suppose that $\alpha \geq 1$ and $\nu >0$. For any $\nu >0$, there exists $\delta_\nu \in (0,1)$
such that $(1-z)^\nu -1 \leq -(\nu/2) z$ for all $z \in [0,\delta_\nu]$. Moreover,
$c(y) \geq c/2 >0$ for all $y \geq y_0$ sufficiently large.
Hence, since the integrand is non-positive, for $x > y_0 / \delta_\nu$,
\begin{align*}
(1+x)^\nu \int_{0}^x \left(   \left(1 -\tfrac{y}{1+x}\right)^\nu - 1 \right) c  (y) y^{-1-\alpha} \ud y & 
\leq \frac{c}{2} (1+x)^\nu  \int_{y_0}^{\delta_\nu x} \left(   \left(1 -\tfrac{y}{1+x} \right)^\nu - 1 \right)   y^{-1-\alpha} \ud y\\
& \leq - \frac{c \nu (1+x)^\nu }{4( 1+x)} \int_{y_0}^{\delta_\nu x}    y^{-\alpha} \ud y , \end{align*}
and part (ii) follows.
\end{proof}

\begin{lemma}
\label{lem:Df}
Suppose that \eqref{ass:tails} holds and  $\chi_i\alpha_i<1$. Then for 
$\nu \in (-1,1 \wedge \alpha_i)$,  $\flat\in \{{\rm one}, {\rm sym} \}$, and $ i \in \cS^\flat$,
as $x\to\infty$,
\[
Df_\nu(x,i)= 
\chi_i\lambda_i c_i x^{\nu-\alpha_i} i_{2,1}^{\nu,\alpha_i} 
 \left(\frac{(P\bla)_i}{\la_i}+R^\flat(\alpha_i,\nu)\right) + o (x^{\nu - \alpha_i} ),
\]
where $\bla=(\lambda_k; k\in\cS)$,  
\begin{align*}
 \Rs(\alpha,\nu)  = \frac{i_0^{\nu,\alpha}+i_1^{\nu,\alpha} -i_{2,0}^{\alpha}}{i_{2,1}^{\nu,\alpha}}, 
\text{ and }  
 \Ro(\alpha,\nu)  =  \frac{\tilde{i}_1^{\nu,\alpha} -i_{2,0}^{\alpha}}{i_{2,1}^{\nu,\alpha}}.
 \end{align*}
 Moreover, if $v_i \in \fD_{\alpha_i,c_i}^+$ then, for some $\delta >0$,  
\[
Df_\nu(x,i)= \chi_i\lambda_i c_i x^{\nu-\alpha_i} i_{2,1}^{\nu,\alpha_i} 
 \left(\frac{(P\bla)_i}{\la_i}+R^\flat (\alpha_i,\nu)\right) + O (x^{\nu - \alpha_i-\delta} ).\]
 Finally,   as $\nu \to 0$, 
 \begin{equation}
\label{eq34} 
R^\flat(\alpha,\nu)= 
\begin{cases}
\Rs (\alpha, \nu)  = -1 + \nu \pi \cot (\pi \alpha_i / 2) + o(\nu)  & \text{ if }\ \ \flat={\rm sym}\\
\Ro (\alpha, \nu)  = -1 + \nu \pi \cot (\pi \alpha_i ) + o(\nu) & \text{ if } \ \ \flat={\rm one}.
\end{cases}
\end{equation}
\end{lemma}
\begin{proof}
The above expressions for $Df_\nu(x,i)$ follow from~\eqref{conditional-increment-sym} and~\eqref{conditional-increment-one}
with Lemma~\ref{lem:all-integrals}. 

Additionally, we compute that 
$\Rs(\alpha,0)=\Ro(\alpha,0)=-1$. For $\nu$ in a neighbourhood of $0$ uniformity of convergence of the integrals over $(1,\infty)$
enables us to differentiate with respect to $\nu$ under the integral sign to get
\begin{align*}
\frac{\partial}{\partial \nu} \Rs (\alpha, \nu)\vert_{\nu=0} &=
\frac{j_0^\alpha +j_1^\alpha}{i_{2,1}^{0,\alpha}} -\Rs(\alpha,0) \frac{j_2^\alpha}{i_{2,1}^{0,\alpha}}\\
& =
\psi \left(1- \tfrac{\alpha}{2} \right)- \psi \left( \tfrac{\alpha}{2} \right) \\
& = \pi\cot \left( \tfrac{\pi \alpha}{2} \right),
\end{align*}
using Lemma~\ref{lem:j-integrals} and the digamma reflection formula (equation 6.3.7 from \cite[p.\ 259]{as}),
and then the first formula in~\eqref{eq34} follows by Taylor's theorem.
Similarly,
for the second formula in~\eqref{eq34}, 
\begin{align*}
\frac{\partial}{\partial \nu} \Ro (\alpha, \nu)\vert_{\nu=0} &=
\frac{\tilde j_1^\alpha}{i_{2,1}^{0,\alpha}} -\Ro(\alpha,0) \frac{j_2^\alpha}{i_{2,1}^{0,\alpha}}\\
& =
\psi(1-\alpha)- \psi(\alpha) \\
& = \pi\cot (\pi \alpha). \qedhere
\end{align*}
\end{proof}

We conclude this subsection with two algebraic results.

\begin{lemma}
\label{lem:finding-theta}
Suppose that \eqref{ass:basic} holds. Given $(b_k ; k \in \cS)$ with $b_k \in \R$ for all $k$, there exists a solution
  $( \theta_k ; k \in \cS)$ with $\theta_k \in \R$ for all $k$ to the system of equations
\begin{equation}
\label{lambda_system}
 \sum_{j \in \cS} p (k,j)  \theta_j - \theta_k = b_k, ~~ ( k \in \cS) , \end{equation}
if and only if $\sum_{k \in \cS} \mu_k b_k = 0$.
Moreover, if a solution to \eqref{lambda_system}  exists, we may take   $\theta_k >0$ for all $k \in \cS$.
\end{lemma}
\begin{proof}
As column vectors, we write $\bmu = (\mu_k ; k \in \cS)$ for the  stationary probabilities as given in \eqref{ass:basic},
 $\bb = (b_k ; k \in \cS)$, and $\bt = (\theta_k ; k \in \cS)$. Then
in matrix-vector form, \eqref{lambda_system} reads $( P - I ) \bt = \bb $,
while $\bmu$ satisfies \eqref{mu_system}, which reads $( P - I)^\tra \bmu = \0$, the homogeneous system adjoint to \eqref{lambda_system}. (Here $I$ is the 
identity matrix and $\0$ is the vector  of all   0s.)

A standard result from linear algebra (a version of the Fredholm alternative)
says that $( P - I ) \bt = \bb $ 
  admits a solution $\bt$ 
if and only if the vector $\bb$ is orthogonal to any solution $\bx$  
to $(P - I)^\tra \bx = \0$; but, by \eqref{ass:basic},
any such $\bx$ is a scalar multiple of $\bmu$. 
In other words, a solution  $\bt$ to \eqref{lambda_system}  exists
if and only if $\bmu^\tra \bb = 0$, as claimed. 

Finally,  since $P$ is a stochastic matrix, $(P - I ) \mathbf{1} = \0$, where $\mathbf{1}$ is the column vector of all 1s;
hence if $\bt$ solves $( P - I ) \bt = \bb $, then so does $\bt + \gamma \mathbf{1}$ for any $\gamma \in \R$.
This implies the final statement in the lemma.
\end{proof}

\begin{lemma}
\label{lem:algebraic2}
Let $U=(U_{k,\ell} ; k,\ell=0,\ldots, M )$ be a given upper triangular matrix having all its upper triangular  elements non-negative 
($U_{k,\ell}\geq 0$ for $0\leq k<\ell\leq M$ and vanishing all other elements) and $A =(A_k ; k=1, \ldots, M )$ a vector with positive components. 
Then there exists a unique lower triangular matrix  $L=(L_{k,\ell} ; k, \ell = 0,\ldots, M)$ 
(so diagonal and upper triangular elements vanish)
satisfying
\begin{itemize}
\item[(i)] $L_{m,m-1} =(UL)_{m,m} +A_m$ for $m=1, \ldots, M$;
\item[(ii)] $L_{k, \ell}=\sum_{r=0}^{k-\ell-1} L_{\ell+r+1, \ell+r}= L_{\ell+1,\ell}+ \cdots + L_{k, k-1}$ for $0\leq \ell < k \leq M$.
\end{itemize}
Also, all lower triangular elements of $L$ are positive, i.e.\ $L_{k,\ell}>0$ for $0\leq \ell<k\leq M$.
\end{lemma}
\begin{proof}
We construct $L$ inductively.
 Item (i) demands
\begin{equation}
\label{eq:L-step}
L_{m,m-1}
=\sum_{\ell=0}^M U_{m,\ell}L_{\ell,m} +A_m=
\sum_{\ell=m+1}^M U_{m,\ell}L_{\ell,m} +A_m.
\end{equation} 
In the case $m=M$, with the usual convention that an empty sum is $0$, the demand~\eqref{eq:L-step} is simply 
$L_{M,M-1}= A_M$.
So we can start our construction taking $L_{M, M-1} = A_M$, which is positive by assumption. 
(Item (ii) makes no demands in the case $k = M$, $\ell = M-1$.)

Suppose now that all matrix elements $L_{k,\ell}$ have been computed in the lower-right corner $\Lambda_m$ ($1 \leq m \leq M$):
\[\Lambda_m \hskip5mm=\hskip5mm \begin{matrix} L_{m,m-1} & & & &\\
		        L_{m+1, m-1} & L_{m+1, m} & &\\
		        \vdots & \vdots & \ddots & \\
		        L_{M,m-1} & L_{M, m} & \ldots & L_{M,M-1}
		        \end{matrix}\]
The elements of $L$ involved in statement (i) (for given $m$) and in statement (ii) for $\ell = m-1$
are all in $\Lambda_m$; thus as part of our inductive hypothesis we may suppose
that the elements of $\Lambda_m$ are such that (i) holds for the given $m$, and (ii) holds
with $\ell = m-1$ and all $m \leq k \leq M$. We have shown that we can achieve this for $\Lambda_M$.

The inductive step is to construct from $\Lambda_m$ ($2 \leq m \leq M$) elements $L_{k, m-2}$ for $m-1 \leq k \leq M$
and hence complete the array $\Lambda_{m-1}$ in such a way that (i) holds for $m-1$ replacing $m$,
 that (ii) holds for $\ell = m-2$, and that all elements are positive.
Now~\eqref{eq:L-step} reveals  the demand of item (i) as
\[  L_{m-1,m-2} = \sum_{\ell=m}^M  U_{m-1,\ell } L_{\ell, m-1} +A_{m-1} ,\]
which we can achieve since the elements of $L$ on the right-hand side are all in $\Lambda_m$,
and since $A_{m-1} > 0$ we get $L_{m-1, m-2} > 0$.

The $\ell = m-2$ case of (ii) demands that for $m \leq k \leq M$ we have
\begin{equation*}
L_{k,m-2} =\sum_{r=0}^{k-m+1} L_{m-1+r, m-2+r}=
L_{m-1, m-2} +\cdots +L_{k,k-1}
\end{equation*}
which involves only elements of $\Lambda_m$ in addition to $L_{m-1, m-2}$, which we have already defined,
and positivity of all the $L_{k, m-2}$ follows by hypothesis.
This gives us the construction of $\Lambda_{m-1}$ and establishes the inductive step.

This algorithm can be continued down to $\Lambda_1$. But then the lower triangular matrix $L$ is totally determined. 
The diagonal and upper triangular elements of $L$ do not influence the construction, and may be set to zero.
\end{proof}

\begin{corollary}
\label{cor:algebraic2-open}
Let the matrix $U$ and the vector $A$ be as in Lemma~\ref{lem:algebraic2}. Let 
$\cL$ be the set of lower triangular matrices $\tilde{L}$ satisfying
\begin{itemize}
\item[(i)] $\tilde{L}_{m,m-1} >(U\tilde{L})_{m,m} +A_m$ for $m=1, \ldots, M$;
\item[(ii)] $\tilde{L}_{k,\ell}=\sum_{r=0}^{k-\ell-1}\tilde{L}_{\ell+r+1, \ell+r}= \tilde{L}_{\ell+1,\ell}+ \cdots + \tilde{L}_{k, k-1}$ for $0\leq \ell < k \leq M$,
\end{itemize}
viewed as subset of the positive cone $\mathcal{V}= (0,\infty)^{\frac{M(M-1)}{2}}$.
Then $\cL$ is a non-empty, open subset of $\mathcal{V}$.
\end{corollary}

\subsection{Supermartingale conditions and recurrence}

 We use the notation
\[ \ubar \alpha := \min_{k \in \cS} \alpha_k ; ~ \bar \alpha := \max_{k \in \cS} \alpha_k ;  ~
\alpha_\star := \min_{k \in \cS} \left\{ \alpha_k \wedge (1/ \chi_k)  \right\}
~\text{and}~
\alpha^\star := \max_{k \in \cS} \left\{ \alpha_k \wedge (1/ \chi_k)  \right\}  .\]

We start with the case  $\max_{k \in \cS} \chi_k \alpha_k < 1$. We will obtain a local supermartingale by choosing the $\la_k$
carefully.  Lemma~\ref{lem:finding-theta}, which shows how the stationary probabilities $\mu_k$ enter, is crucial;
a similar idea was used for random walks on strips in Section~3.1 of~\cite{fmm}. 
Next   is our key local  supermartingale result in this case.

\begin{proposition}
\label{pro:main-supermartingale}
Suppose that \eqref{ass:basic} and \eqref{ass:tails} hold, and that $\max_{k \in \cS} \chi_k \alpha_k < 1$. 
\begin{itemize}
\item[(i)]
If $\sum_{k \in \cS} \mu_k \cot ( \chi_k \pi \alpha_k ) < 0$, then  there exist $\nu \in (0, \ubar \alpha)$, $\la_k >0$ ($k \in \cS$), $\eps >0$, and $x_0 \in \RP$ such that
\[ \hskip -5mm Df_\nu(x,i)= \Exp_{x,i} [ f_\nu ( X_1, \xi_1 ) - f_\nu ( X_0, \xi_0 ) ] \leq - \eps x^{\nu - \bar \alpha} ,  ~\text{for all } x \geq x_0 \text{ and all } i \in \cS .\]      
\item[(ii)]
If $\sum_{k \in \cS} \mu_k \cot ( \chi_k \pi \alpha_k ) > 0$, then  there exist $\nu \in (-1,0)$, $\la_k >0$ ($k \in \cS$), $\eps >0$, and $x_0 \in \RP$ such that
\[ \hskip -5mm Df_\nu(x,i)= \Exp_{x,i} [ f_\nu ( X_1, \xi_1 ) - f_\nu ( X_0, \xi_0 ) ] \leq - \eps x^{\nu - \bar \alpha} ,  ~\text{for all } x \geq x_0 \text{ and all } i \in \cS .\]  
\end{itemize}    
\end{proposition}
\begin{proof}
Write $a_k = \pi \cot ( \chi_k \pi \alpha_k )$, $k \in \cS$. First we prove part~(i). 
By hypothesis, 
$\sum_{k \in \cS} \mu_k a_k = - \delta$ for some $\delta > 0$.
Set $b_k = -a_k - \delta$ for all $k \in \cS$, so that $\sum_{k \in \cS} \mu_k b_k = 0$.
Lemma~\ref{lem:finding-theta} shows that for these $b_k$, we can find a collection of $\theta_k \in (0,\infty)$ so that
\begin{equation}
\label{eq:theta-rec}
 \sum_{j \in \cS} p (i,j)  \theta_j - \theta_i + a_i = - \delta , ~\text{for all}~ i \in \cS ;\end{equation}
fix these $\theta_k$ for the rest of the proof.
We then choose 
$\bla=(\la_k ; k\in\cS )$ of the form  $\la_k:=\la_k(\nu) = 1 + \theta_k \nu$, for some
$\nu \in (0, 1 \wedge \ubar \alpha)$.
 
Since $i_{2,1}^{\nu,\alpha}>0$ for all $\alpha>0$ and all $\nu\in(0,\alpha)$,  Lemma~\ref{lem:Df} shows that
$Df_\nu(x,i)$ will be negative for all $i$ and all $x$ sufficiently large 
provided that we can find $\nu$ such that
$\frac{(P\bla)_i}{\la_i}+R^\flat (\alpha_i,\nu) < 0$ for all $i$.
By~\eqref{eq34}, writing $\bt = (\theta_k ; k \in \cS)$,
 \begin{align*}
 \frac{(P\bla)_i}{\la_i}+R^\flat (\alpha_i,\nu) & = \frac{(P\bla)_i}{\la_i} -1 + \nu a_i +o(\nu) \\
& = \frac{1+ \nu (P\bt)_i}{1+\nu \theta_i} -1 + \nu a_i +o(\nu)\\
 &= \nu( (P\bt)_i- \theta_i) + \nu a_i   +o(\nu) = -\nu \delta +o(\nu),
 \end{align*}
by~\eqref{eq:theta-rec}.
Therefore, by Lemma~\ref{lem:Df}, we can always find sufficiently small $\nu > 0$ and a vector $\bla=\bla(\nu)$ with strictly positive elements
 for which 
$Df_\nu(x,i) \leq -\eps x^{\nu - \alpha_i}$ for some $\eps >0$, all $i$, and all $x$ sufficiently large. Maximizing over $i$ gives part~(i).

The argument for part~(ii) is similar. Suppose $\nu \in (-1,0)$.
This time,  $\sum_{k \in \cS} \mu_k a_k = \delta$ for some $\delta > 0$,
and we set $b_k = -a_k + \delta$, so that $\sum_{k \in \cS} \mu_k b_k = 0$ once more.
Lemma~\ref{lem:finding-theta} now shows that  we can find $\theta_k$ so that
\begin{equation}
\label{eq:theta-tran}
 \sum_{j \in \cS} p (i,j)  \theta_j - \theta_i + a_i = \delta , ~\text{for all}~ i \in \cS .\end{equation}
With this choice of $\theta_k$ we again set $\lambda_k = 1 +  \theta_k\nu$; note we may assume
$\lambda_k >0$ for all $k$ for $\nu$ sufficiently small.

Again,
$Df_\nu(x,i)$ will be non-negative for all $i$ and all $x$ sufficiently large 
provided that we can find $\bla$ and $\nu$ such that
$\frac{(P\bla)_i}{\la_i}+R^\flat (\alpha_i,\nu) < 0$ for all $i$.
Following a similar argument to before,  we obtain with~\eqref{eq:theta-tran} that
\begin{align*}
 \frac{(P\bla)_i}{\la_i}+R^\flat (\alpha_i,\nu) & = \nu( (P\bt)_i- \theta_i) + \nu a_i   +o(\nu) =  \nu \delta +o(\nu).
 \end{align*}
Thus we can find for $\nu < 0$ close enough to $0$ a vector $\bla=\bla(\nu)$ with strictly positive elements 
for which $Df_\nu(x,i)\leq -\eps x^{\nu - \alpha_i}$ for all $i$ and all $x$ sufficiently large.
\end{proof}

Now we examine the case $\max_{k \in \cS} \chi_k\alpha_k \geq 1$.

\begin{proposition}
\label{pro:recurrence-i}
Suppose that \eqref{ass:basic} and \eqref{ass:tails} hold, and $\max_{k \in \cS} \chi_k \alpha_k \geq 1$. 
Then there exist $\nu \in (0, \alpha_\star)$, $\la_k >0$ ($k \in \cS$), $\eps >0$, and $x_0 \in \RP$ such that for all $i\in\cS$,
\begin{equation}
\label{eq:drift-star}
 Df_\nu(x,i)=\Exp_{x,i} [ f_\nu ( X_1, \xi_1 ) - f_\nu ( X_0, \xi_0 ) ] \leq - \eps x^{\nu - \alpha^\star} ,  ~\text{for all } x \geq x_0.\end{equation}
\end{proposition}
Before starting the proof of this proposition, we introduce the following notation. For $k\in\cS$ denote by  $a_k=\pi \cot(\pi \chi_k \alpha_k)$ and define the vector $\ba=(a_k ; k\in\cS)$.
 For $i\in\cS$ and $A \subseteq \cS$, write $P(i,A) =\sum_{j\in A} p(i,j)$.
 Define $\cS_0=\{i\in\cS: \chi_i \alpha_i \geq 1\}$ and recursively, for $m \geq 1$, 
\[\cS_m:= \left\{ i \in\cS\setminus \cup_{\ell=0}^{m-1} \cS_\ell: P(i,\cS_{m-1})>0 \right\}. \]
Denote by $M:=\max \{m\geq 0: \cS_m\ne \emptyset\}$. Since $P$ is irreducible,  the collection
$(\cS_m ; m=0,\ldots, M )$ is a partition of $\cS$. 

\begin{proof}[Proof of Proposition \ref{pro:recurrence-i}]
It suffices to find $x_0\in\R_+$, $\eps >0$, $\nu \in (0, \alpha_\star)$, and an open, non-empty subset $G$ of the positive cone $\mathcal{C}:= (0,\infty)^{|\cS|}$ such that 
\begin{equation*}
G\subseteq \cap_{i\in\cS} \{\bla\in\mathcal{C}: Df_\nu(x,i)  \text{ satisfies condition (\ref{eq:drift-star})}. \}
\end{equation*}
Now for $i\in\cS_0$, inequality~\eqref{eq:drift-star} is satisfied thanks to \eqref{conditional-increment-sym}, 
\eqref{conditional-increment-one} and 
Lemmas~\ref{lem:all-integrals} and~\ref{lem:big-alpha} 
for every choice of $\bla$ (with positive components) and $\nu \in (0, \ubar \alpha)$.
 Hence the previous condition  reduces to the requirement 
\begin{equation}
\label{eq:open-modified}
G\subseteq \cap_{i\in\cS\setminus \cS_0} \{\bla\in\mathcal{C}: Df_\nu(x,i)  \text{ satisfies condition \eqref{eq:drift-star}}.\}
\end{equation}
The rest of the proof is devoted into establishing this fact.

Suppose that $i\in\cS_m$ with $m=1,\ldots, M$. 
Then $\chi_i \alpha_i < 1$ by construction, and Lemma~\ref{lem:Df} shows that
condition \eqref{eq:drift-star} will be satisfied if the system of inequalities
\begin{equation}
\label{eq:inequalities}
\frac{(P\bla)_i}{\la_i} +R^\flat(\alpha_i, \nu)<0, ~ i=1, \ldots, M
\end{equation} have non-trivial solutions $\bla$ for sufficiently small $\nu$. 
Thanks to the  Lemma~\ref{lem:Df}, we have  $R^\flat (\alpha_i,\nu)= -1+\nu a_i +o(\nu)$. We will obtain~\eqref{eq:drift-star}
if, for  $\nu$ sufficiently small,  
\begin{equation}
\label{eq:drift-needed}
\cR=\cR(P, \nu,\ba):=\bigcap_{m=1}^M \left\{\bla\in\mathcal{C}: \frac{(P\bla)_i}{\la_i}  < 1 - \nu a_i, i\in\cS_m\right\}\ne\emptyset.
\end{equation}
We seek  a  solution $\bla\in\cR$  (for sufficiently small $\nu$),  under the Ansatz that the $\la_j$
 are constant on every $\cS_\ell$, i.e., the vector 
$\bla$ has the form $\hat{\bla}$ with  $\hat{\la}_j=\la^{(\ell)}$  for all $j\in\cS_\ell$.
Suppose that $i \in \cS_m$. Then  $p(i,j) = 0$ for $j \in \cS_\ell$ with $\ell < m-1$, so that
\begin{equation}
\label{eq:Plambda}
\frac{(P\hat{\bla})_i}{\hat{\la}_i}= \sum_{j\in\cS} p(i,j) \frac{\hat{\la}_j}{\hat{\lambda}_i} 
= \sum_{\ell=0}^M P(i,\cS_\ell)
\frac{\lambda^{(\ell)}}{\la^{(m)}} = \sum_{\ell=m-1}^MP(i,\cS_\ell)
\frac{\lambda^{(\ell)}}{\la^{(m)}}.
\end{equation}
We introduce the auxiliary matrix 
$\rho =(\rho_{k,\ell} ; k,\ell \in \{0, \ldots, M\} )$ defined by $\rho_{k,\ell}:= \la^{(k)}/ \la^{(\ell)}$. 
By construction,  $\rho_{k,k}=1$ and $\rho_{k,\ell}= 1/\rho_{\ell,k}$. Let 
\begin{equation}
\label{eq:L-def}
 L_{k, \ell} =  \frac{1}{\nu} \log \rho_{k,\ell} = - \frac{1}{\nu} \log \rho_{\ell, k} = \frac{1}{\nu} \left( \log \lambda^{(k)} - \log \lambda^{(\ell)} \right) .
\end{equation}
It suffices to determine the upper triangular part of $\rho$, or, equivalently,
the lower triangular array $(L_{k, \ell} ; 0 \leq \ell < k \leq M )$.
We do so recursively, starting with $L_{M, M-1}$. In the case $i\in\cS_M$, the condition
in~\eqref{eq:drift-needed} reads, by~\eqref{eq:Plambda},
\[\rho_{M-1, M} < \frac {1-P(i,\cS_M) -\nu a_i}{P(i,\cS_{M-1})}= \frac{P(i, \cS_M^\rc)}{P(i,\cS_{M-1})} -\nu \frac{a_i}{P(i,\cS_{M-1})}= 1-\nu \frac{a_i}{P(i,\cS_{M-1})}.\]
On introducing the constant $A_M=\max_{i\in\cS_M} \frac{a_i}{P(i,\cS_M)}$, it is enough to choose $\rho_{M-1,M} < 1-\nu A_M= \exp(-\nu A_M) +o(\nu)$. 
In other words, with $\rho_{\ell,k}= \exp(-\nu L_{k,\ell})$, we see that the choice $L_{M, M-1} >A_M$ satisfies the condition 
in~\eqref{eq:drift-needed} for $i\in \cS_M$.
 
 Suppose now that we have determined the condition in~\eqref{eq:drift-needed} for $i\in \cup_{\ell = m}^M \cS_\ell$. 
 Then, for $i \in \cS_{m-1}$ the condition amounts, via~\eqref{eq:Plambda}, to 
 \[
 \rho_{m-1, m} P(i, \cS_{m-1})  < 1- \nu a_i -P(i, \cS_m) - \sum_{\ell=m+1}^M \rho_ {\ell,m} P(i, \cS_\ell).\]
Using the fact that for $\ell < m$ we have 
$\rho_{\ell,m} = \exp(-\nu L_{m, \ell})$ and for $\ell > m$ we have $\rho_{\ell , m} = \exp(\nu L_{\ell, m})$, the above expression becomes, up to $o(\nu)$ terms
\begin{align*}
1-\nu L_{m,m-1} & < \frac{P(i, (\cup_{\ell=m}^M \cS_\ell)^c) }{P(i,\cS_{m-1})} -\nu \sum _{\ell=m+1}^M
\frac{P(i, \cS_l)}{P(i, \cS_{m-1})} L_{\ell,m} - \nu \frac{a_i}{P(i, \cS_{m-1})}\\
&= 1- \nu \sum _{\ell=m+1}^M
\frac{P(i, \cS_\ell)}{P(i, \cS_{m-1})} L_{\ell,m} - \nu \frac{a_i}{P(i, \cS_{m-1})}.
\end{align*}
Introducing the upper triangular matrix $U=(U_{m,n} ;  0 \leq m<n\leq M )$ defined
 by $U_{m,n}= \max_{i\in \cS_m} \frac{P(i, \cS_n)}{P(i, \cS_{m-1})}$ for $m\geq1$,
and the vector $A_m=\max_{i\in\cS_m} \frac{a_i}{P(i, \cS_{m-1})}$ for $m=1, \ldots, M$,
the condition in~\eqref{eq:drift-needed}  is satisfied if we solve the recursion
\[L_{m,m-1} >(UL)_{m, m} +A_m, \ \text{ for } m=M, M-1, \ldots, 1,\]
with initial condition $L_{M,M-1}>A_M$ and condition $L_{\ell,m} >0$ for $0 \leq m < \ell \leq M$.
 Additionally, we have from~\eqref{eq:L-def} that
\[ L_{k, \ell} = L_{k, k-1} + L_{k-1, k-2} + \cdots + L_{\ell +1, \ell}, ~~ 0 \leq \ell < k \leq M . \] 
Hence by Corollary \ref{cor:algebraic2-open}, there exist non-trivial solutions for the lower triangular matrix
$L$ within an  algorithmically determined region $\mathcal{L}$. The positivity of the lower triangular part of $L$ implies that the components of $\bla$ are ordered: $\la^{(m)}<\la^{(m+1)}$ for $0\leq m<M-1$. By choosing $\lambda^{(M)}$ sufficiently large, we can guarantee that the first term  (and consequently the entire sequence) satisfies $\lambda^{(0)}>1$. 
\end{proof}

We are almost ready to complete the proof of Theorem~\ref{thm:recurrence}, excluding part (b)(iii); first we need one more technical result concerning non-confinement.

\begin{lemma}
\label{l:orw_non-confinement}
Suppose that~\eqref{ass:basic} and~\eqref{ass:tails} hold.  Then $\limsup_{n \to \infty}  X_n = \infty$, a.s.
\end{lemma}
\begin{proof}
We claim that for each $x \in \RP$,
there exists $\eps_x >0$ such that
\begin{equation}
\label{e:local_escape_orw}
 \Pr [  X_{n+1}  -  X_n  \geq 1 \mid ( X_n , \xi_n) = (y,i)  ] \geq \eps_x, ~\text{for all}~ y \in [0,x] \text{ and all } i\in \cS .
\end{equation}
Indeed, given $(x,i) \in \RP \times \cS$, we may choose $j \in \cS$
so that $p(i,j) >0$ and we may choose $z_0 \geq x$ sufficiently large so that, for some $\eps >0$, $v_i (z) \geq \eps z^{-1-\alpha_i}$
 for all $z \geq z_0$. Then if $y \in [0,x]$,
\begin{align*}
\Pr [ X_{n+1} \geq y + 1  \mid (X_n , \xi_n) = (y,i) ] & \geq p(i,j) \int_{2z_0+1}^\infty \eps z^{-1-\alpha_i} \ud z = \eps_{x, i} >0 , \end{align*}
which gives~\eqref{e:local_escape_orw}.
The local escape property~\eqref{e:local_escape_orw} implies the $\limsup$ result by a standard argument.
\end{proof}

\begin{proof}[Proof of Theorem~\ref{thm:recurrence}.]
We are not yet ready to prove part (b)(iii): we defer that part of the proof until Section \ref{sec:critical}.

The other parts of the theorem follow from the supermartingale estimates in this section together with the
technical results from Section~\ref{sec:criteria}. Indeed, under the conditions of part (a) or (b)(i)
of the theorem, we have from Proposition~\ref{pro:recurrence-i} or Proposition~\ref{pro:main-supermartingale}(i) respectively that
for suitable $\nu >0$ and $\la_k$,  
\[ \Exp [ f_\nu ( X_{n+1}, \xi_n ) - f_\nu (X_n , \xi_n ) \mid X_n , \xi_n ] \leq 0 , ~~\text{on}~ \{ X_n \geq x_0\} .\]
Thus we may apply Lemma~\ref{l:adapted_rec_R}, which together with Lemma~\ref{l:orw_non-confinement} shows that
$\liminf_{n \to \infty} X_n \leq x_0$, a.s. Thus there exists an interval $I \subseteq [0,x_0+1]$
such that 
  $(X_n,\eta_n) \in I \times \{ i \}$ i.o., where $i$ is some fixed element of $\cS$. 
Let $\tau_0 := 0$ and for $k \in \N$ define $\tau_k = \min \{ n > \tau_{k-1}+1 : (X_n, \eta_n) \in I \times \{ i \} \}$.
Given $i \in \cS$, we may choose $j, k \in \cS$ such that $p(i,j) > \delta_1$ and $p(j,k) > \delta_1$ for some $\delta_1 >0$;
let $\gamma = \alpha_i \vee \alpha_j$. Then we may choose $\delta_2 \in (0,1)$ and  $z_0 \in \RP$ such that 
$v_i (z) > \delta_2 z^{-1-\gamma}$
and
$v_j (z ) > \delta_2 z^{-1-\gamma}$ for all $z \geq z_0$. Then for any $\eps \in (0,1)$,
\begin{align*} 
\Pr [   X_{\tau_k +2 }   < \eps \mid X_{\tau_k} ] & \geq \Pr [  X_{\tau_k +2 }  < \eps ,  X_{\tau_k + 1} \in [z_0 +1, z_0 +2 ] , \eta_{\tau_k+1} = j \mid X_{\tau_k} ] \\
& \geq \delta_1^2 \delta_2^2  \eps (z_0 +3)^{-1-\gamma} (x_0 + z_0 + 3)^{-1-\gamma} ,\end{align*}
uniformly in $k$. Thus L\'evy's extension of the Borel--Cantelli lemma  
shows $X_n < \eps$ infinitely often. Thus, since $\eps \in (0,1)$ was arbitrary, $\liminf_{n \to \infty}  X_n = 0$, a.s.

On the other hand, under the conditions of part   (b)(ii)
of the theorem, we have from  Proposition~\ref{pro:main-supermartingale}(ii)  that
for suitable $\nu <0$ and $\la_k$,  
\[ \Exp [ f_\nu ( X_{n+1}, \xi_n ) - f_\nu (X_n , \xi_n ) \mid X_n , \xi_n ] \leq 0 , ~~\text{on}~ \{ X_n \geq x_1\} ,\]
for any $x_1$ sufficiently large.
Thus we may apply Lemma~\ref{l:adapted_good_escape_R}, which shows that
for any $\eps >0$ there exists $x \in (x_1, \infty)$
  for which, for all $n \geq 0$,
\[ \Pr \Bigl[ \inf_{m \geq n} X_m \geq x_1 \Bigmid \cF_n \Bigr] \geq 1 - \eps, \text{ on } \{ X_n \geq x \} . \]
Set $\sigma_x = \min \{ n \geq 0 : X_n  \geq x \}$. 
 Then, on $\{ \sigma_x < \infty \}$,
\[ \Pr \Bigl[   \inf_{m \geq \sigma_x} X_m > x_1 \Bigmid \cF_{\sigma_x} \Bigr] \geq 1 - \eps, \as \]
But on $\{ \sigma_x < \infty \} \cap \{ \inf_{m \geq \sigma_x} X_m > x_1 \}$
we have $\liminf_{m \to \infty} X_m \geq x_1$, so
\begin{align*}
\Pr \Bigl[  \liminf_{m \to \infty} X_m \geq x_1 \Bigr] & \geq 
\Exp \left[ \Pr \Bigl[   \inf_{m \geq \sigma_x} X_m > x_1 \Bigmid \cF_{\sigma_x} \Bigr] \1{ \sigma_x < \infty  } \right] \\
& \geq (1 -\eps ) \Pr [ \sigma_x < \infty ] = (1-\eps),\end{align*}
by Lemma~\ref{l:orw_non-confinement}. Since $\eps >0$ was arbitrary, we get $\liminf_{m \to \infty} X_m \geq x_1$, a.s.,
and since $x_1$ was arbitrary we get $\lim_{m \to \infty}  X_m  = \infty$, a.s.
\end{proof}

\section{Existence or non-existence of moments}
\label{sec:moments}

\subsection{Technical tools}

The following result is a straightforward reformulation of Theorem~1 of~\cite{aim}.

\begin{lemma}
\label{lem:AIM}
Let $Y_n$ be  an integrable $\cF_n$-adapted stochastic process, taking values in an
unbounded subset of~$\RP$, with $Y_0 = x_0$ fixed.
For $x >0$, let $\sigma_x := \inf \{ n \geq 0 : Y_n \leq x \}$.
Suppose that there exist $\delta>0$, $x >0$, and $\gamma < 1$ such that for
any $n \geq 0$,
\begin{equation}
\label{eq:AIM1}
\Exp [ Y_{n+1}-Y_n \mid \cF_n ] \leq -\delta Y_n^{\gamma},  
 \text{ on } \{n < \sigma_x \}.
\end{equation}
Then, for any $p \in [0,1/(1-\gamma))$, 
$\Exp [ \sigma_x^p ]  
< \infty$.
\end{lemma}

The following companion result on non-existence of moments is a reformulation of Corollary~1 of~\cite{aim}.

\begin{lemma}
\label{lem:AIM2}
Let $Y_n$ be  an integrable $\cF_n$-adapted stochastic process, taking values in an
unbounded subset of~$\RP$, with $Y_0 = x_0$ fixed.
For $x >0$, let $\sigma_x := \inf \{ n \geq 0 : Y_n \leq x \}$.
Suppose that there exist $C_1, C_2 >0$, $x >0$, $p >0$ and $r > 1$ such that for
any $n \geq 0$, on $\{ n < \sigma_x \}$ the following hold:
\begin{align}
\label{eq:AIM2a}
\Exp [ Y_{n+1}-Y_n \mid \cF_n ] & \geq - C_1 ; \\
\label{eq:AIM2b}
\Exp [ Y^r_{n+1}-Y^r_n \mid \cF_n ] & \leq C_2 Y_n^{r-1} ; \\
\label{eq:AIMc}
\Exp [ Y_{n+1}^p - Y_n^p \mid \cF_n ] & \geq 0 
.\end{align} 
Then for any $q >p$,  $\Exp [ \sigma_x^q ] = \infty$ for $x_0 > x$.
\end{lemma}

\subsection{Proof of Theorem~\ref{thm:moments}}

\begin{proof}[Proof of Theorem~\ref{thm:moments}.]
Under conditions (a) or (b)(i) of Theorem~\ref{thm:recurrence},
we have from Propositions~\ref{pro:recurrence-i} or~\ref{pro:main-supermartingale} respectively that
there exist positive $\lambda_k$ and constants $\eps >0$, $\beta >0$ and $\nu \in (0,\beta)$ such that,
\[ D f_\nu (x,i) \leq - \eps x^{\nu - \beta} , \text{ for all } x \geq x_0 \text{ and all } i.\]
Let $Y_n = f_\nu ( X_n, \xi_n)$. Then  $Y_n$ is bounded above and below by positive constants times $(1+X_n)^\nu$,
so we have that~\eqref{eq:AIM1} holds for $x$ sufficiently large with $\gamma = 1 - (\beta / \nu)$.
It follows from Lemma~\ref{lem:AIM} that $\Exp [ \sigma_x^p ]  
< \infty$
for $p \in (0, \nu / \beta)$, which gives the claimed existence of moments result.

It is not hard to see that \emph{some} moments of the return time fail to exist, due to the heavy-tailed nature of the model, and an argument is easily
constructed using the `one big jump' idea: a similar idea is used in~\cite{hmmw}. We sketch the argument. For any $x, i$, for all $y$ sufficiently large 
we have $\Pr_{x,i} [ X_1 \geq  y-x ] \geq \eps y^{-\bar \alpha}$. Given such a first jump, with uniformly positive probability
the process takes time at least of order $y^\beta$ to return to a neighbourhood of zero (where $\beta$ can be bounded in terms of $\ubar \alpha$);
this can be proved using a suitable maximal inequality as in the proof of Theorem~2.10 of~\cite{hmmw}.
Combining these two facts shows that with   probability of order $y^{-\bar \alpha}$ the return time to a neighbourhood of the origin exceeds
order $y^\beta$. This polynomial tail bound yields non-existence of sufficiently high moments.
\end{proof}

\subsection{Explicit cases: Theorems~\ref{thm:symmetric_random_walk_moments} and~\ref{thm:one-sided_moments}}

We now restrict attention to the  case $\cS = \{ 1,2 \}$ with $\alpha_1 = \alpha_2 = \alpha$
and $\chi_1 = \chi_2 = \chi$, so both half-lines are of the same type.
Take $\lambda_1 = \lambda_2 = 1$ and $\nu \in (0,  \alpha)$, so that $f_\nu (x , i ) = (1 +x)^\nu$.
Then Lemma~\ref{lem:Df} shows that, for $i \in \cS^\flat$, $\flat \in \{ {\rm sym}, {\rm one} \}$,
\begin{align}
\label{eq:D-C}
 D f_\nu (x,i) & = \chi c_i x^{\nu - \alpha} C^\flat (\nu, \alpha ) + o (x^{\nu -\alpha} ) ,\end{align}
where
\begin{align} C^{\rm sym} (\nu , \alpha ) & =  i^{\nu, \alpha}_{2,1} + i^{\nu,\alpha}_0 +  i^{\nu,\alpha}_1 - i^\alpha_{2,0}  ; \nonumber\\
 C^{\rm one} (\nu , \alpha ) & =  i^{\nu, \alpha}_{2,1} + \tilde i^{\nu,\alpha}_1 - i^\alpha_{2,0}  . \nonumber
\end{align}

The two cases we are interested in are the recurrent two-sided symmetric case, where $\chi =\frac{1}{2}$ (i.e., $\cS = \Ss$) with $\alpha > 1$,
and the recurrent one-sided antisymmetric case, where $\chi=1$ (i.e., $\cS = \So$) with $\alpha > \frac{1}{2}$.

\begin{lemma}
\label{lem:C-properties}
Let $\flat \in \{ {\rm sym}, {\rm one} \}$ and $\chi \alpha \in (\frac{1}{2},1)$.
The function $\nu \mapsto C^\flat (\nu,\alpha)$ is continuous for $\nu \in [0, \alpha)$ with $C (0,\alpha) =0$
and $\lim_{\nu \uparrow \alpha} C^\flat (\nu, \alpha) = \infty$.
There exists $\nu^\flat_0 = \nu^\flat_0 (\alpha) \in (0,\alpha)$ such that $C^\flat (\nu ,\alpha) < 0$ for $\alpha \in (0, \nu^\flat_0)$,
$C^\flat (\nu^\flat_0, \alpha) = 0$, and $C^\flat (\nu ,\alpha) > 0$ for $\alpha \in (\nu^\flat_0 , \alpha)$
\end{lemma}
\begin{proof}
We give the proof only in the case $\flat = {\rm sym}$; the other case is very similar. Thus $\chi = \frac{1}{2}$, and,
for ease of notation, we write just $C$ instead of $C^\flat$.

Clearly $C ( 0 ,\alpha ) = 0$.
For $\nu \geq 1$, convexity of the function $z \mapsto z^\nu$ on $\RP$ shows that $(1+u)^\nu + (1-u)^\nu - 2 \geq 0$ for all $u \in [0,1]$,
so that $i_1^{\nu, \alpha} \geq 0$; clearly, $i^{\nu, \alpha}_{2,1}$ and
$i_0^{\nu, \alpha}$ are also  non-negative.
 Hence,  by the expression for $i_{2,1}^{\nu,\alpha}$ in Lemma~\ref{lem:i-integrals},
\[ \liminf_{\nu \uparrow \alpha} C (\nu ,\alpha) \geq - \frac{1}{\alpha} + \liminf_{\nu \uparrow \alpha} \frac{\Gamma (\nu +1) \Gamma (\alpha - \nu)}{\Gamma (1+\alpha)} ,\]
which is $+\infty$. Moreover, by Lemma~\ref{lem:j-integrals} and the subsequent remark,
\[ \left. \frac{\partial}{\partial \nu} C (\nu, \alpha ) \right|_{\nu =0} = 
\frac{1}{\alpha} \left( \psi (1 -\tfrac{\alpha}{2} ) - \psi ( \tfrac{\alpha}{2} ) \right) = \frac{\pi}{\alpha}
\cot ( \tfrac{\pi \alpha}{2} ) ,\]
which is negative for $\alpha \in (1,2)$. Hence $C (\nu , \alpha ) <0$ for $\nu > 0$ small enough.

Since $\nu \mapsto C (\nu, \alpha)$ is a non-constant analytic function on $[0,\alpha)$,
its zeros can accumulate only at $\alpha$; but this is ruled out by the fact that $C (\nu ,\alpha) \to \infty$
as $\nu \to \alpha$. Hence $C ( \, \cdot \, ,\alpha)$ has only finitely many zeros in $[0,\alpha)$;
one is at $0$, and there must be at least one zero in $ (0 , \alpha)$, by Rolle's theorem.
Define $\nu_-  := \nu_- (\alpha)$ and $\nu_+ := \nu_+ (\alpha)$
to be the smallest and largest such zeros, respectively.

Suppose $0 < \nu_1 \leq \nu_2 < \alpha$.
By Jensen's inequality, $\Exp_{x,i} [ (1 + X_{1} )^{\nu_2} ]    \geq \left( \Exp_{x,i} [ (1+ X_1 )^{\nu_1} ] \right)^{\nu_2/\nu_1}$. 
Hence
\begin{align*} (1+ x)^{\nu_2} + D f_{\nu_2} ( x, i) & \geq \left( (1+x)^{\nu_1} + D f_{\nu_1} (x, i) \right)^{\nu_2/\nu_1} \\
& = (1+x)^{\nu_2} \left( 1 + (1+x)^{-\nu_1} D f_{\nu_1} (x, i) \right)^{\nu_2/\nu_1} \\
& = (1+x)^{\nu_2} + \frac{\nu_2}{\nu_1}   x^{\nu_2-\nu_1} D f_{\nu_1} (x, i) + o (x^{\nu_2 - \alpha}), \end{align*}
using Taylor's theorem and the fact that $D f_{\nu_1} (x, i) = O (x^{\nu_1 -\alpha} )$, by~\eqref{eq:D-C}.
By another application of~\eqref{eq:D-C}, it follows that 
\[   \chi c_i x^{\nu_2 - \alpha} C (\nu_2, \alpha )  + o (x^{\nu_2 - \alpha})
\geq   \frac{\nu_2}{\nu_1} \chi c_i x^{\nu_2 - \alpha} C (\nu_1, \alpha) + o (x^{\nu_2 - \alpha}) .\]
Multiplying by $x^{\alpha-\nu_2}$ and taking $x \to \infty$, we obtain
\[ C (\nu_2, \alpha ) \geq \frac{\nu_2}{\nu_1} C (\nu_1, \alpha) , \text{ for } 0 < \nu_1 \leq \nu_2 < \alpha.\]
In particular, (i) if 
$C ( \nu_1 ,\alpha ) \geq 0$ then $C ( \nu , \alpha) \geq 0$ for all $\nu \geq \nu_1 >0$; and (ii) if
$C( \nu_1, \alpha) > 0$ and $\nu_2 > \nu_1$, we have $C (\nu_2, \alpha ) > C (\nu_1, \alpha)$.
It follows from these two observations  that $C ( \nu , \alpha) = 0$ for $\nu \in [ \nu_- , \nu_+ ]$, which is not possible
unless $\nu_- = \nu_+$. Hence there is exactly one zero of $C ( \, \cdot \, ,\alpha)$ in $(0,\alpha)$; call it $\nu_0 (\alpha)$.
\end{proof}

\begin{lemma}
\label{lem:nu0}
The positive zero of $C^\flat ( \, \cdot \, , \alpha)$ described in Lemma~\ref{lem:C-properties}
is given by $v^{\rm one}_0 ( \alpha ) = 2 \alpha -1$ or
$v^{\rm sym}_0 (\alpha ) = \alpha -1$.
\end{lemma}
\begin{proof}
First suppose $\flat = {\rm one}$. Then from Lemma~\ref{lem:i-integrals} we verify that for $\alpha \in ( \frac{1}{2}, 1 )$,
\[ C^{\rm one} ( 2 \alpha -1 , \alpha) = \frac{ \Gamma (2\alpha ) \Gamma ( 1 - \alpha)}{\Gamma (1 + \alpha) } - \frac{\Gamma (2\alpha) \Gamma ( 1-\alpha)}{\alpha \Gamma (\alpha ) } = 0 ,\]
since $\alpha \Gamma (\alpha ) = \Gamma (1 +\alpha)$.

Now suppose that $\flat = {\rm sym}$ and $\alpha \in (1,2)$. To verify $C^{\rm sym} (\alpha -1 ,\alpha) = 0$
 it is simpler to work with the integral representations directly, rather than the hypergeometric functions.
After the substitution $z = 1/u$, we have
\[ i^{\alpha -1,\alpha}_0 = \int_0^1 \left( ( 1 + z)^{\alpha -1} - z^{\alpha -1} \right) \ud z = \frac{2^\alpha -2}{\alpha} .\]
Similarly, after the same substitution, 
\[ i_1^{\alpha-1 , \alpha} = \int_1^\infty \left( (z+1)^{\alpha -1} + (z-1)^{\alpha -1} - 2 z^{\alpha -1} \right) \ud z, \]
which we may evaluate as
\begin{align*} i_1^{\alpha-1, \alpha} & = \lim_{y\to \infty} \int_1^y\left( (z+1)^{\alpha -1} + (z-1)^{\alpha -1} - 2 z^{\alpha -1} \right) \ud z \\
& = \frac{2-2^\alpha}{\alpha} + \frac{1}{\alpha} \lim_{y \to \infty} y^\alpha \left( (1+y^{-1} )^\alpha + (1 -y^{-1} )^\alpha - 2  \right) \\
& = \frac{2-2^\alpha}{\alpha} .\end{align*}
Finally we have that  $i^{\alpha-1, \alpha}_{2,1} - i^\alpha_{2,0} = \frac{ \Gamma (\alpha)}{\Gamma (1+\alpha)} - \frac{1}{\alpha} =0$, 
so altogether we verify that $C^{\rm sym} (\alpha -1 ,\alpha ) = 0$.
\end{proof}

We can now complete the proofs of Theorems~\ref{thm:symmetric_random_walk_moments} and~\ref{thm:one-sided_moments}.

\begin{proof}[Proof of Theorem~\ref{thm:symmetric_random_walk_moments}.]
Let $Y_n = f_\nu (X_n, \xi_n )$. 
First suppose that $\alpha \in (1,2)$. Then we have from~\eqref{eq:D-C} together
with Lemmas~\ref{lem:C-properties} and~\ref{lem:nu0} that,
for any $\nu \in (0, \alpha - 1)$,
\[ \Exp [ Y_{n+1} - Y_n \mid \cF_n ] \leq - \eps Y_n^{1-(\alpha/\nu)} , \text{ on } \{ Y_n \geq y_0 \} , \]
for some $\eps>0$ and $y_0 \in \RP$. It follows from Lemma~\ref{lem:AIM} that
$\Exp [ \sigma^p ] < \infty$ for $p < \nu / \alpha$ and since $\nu < \alpha - 1$ was arbitrary we get
$\Exp [ \sigma^p ] < \infty$ for $p < 1 - (1/\alpha)$.

For the non-existence of moments when $\alpha \in (1,2)$, we will apply Lemma~\ref{lem:AIM2} with $Y_n = f_\nu (X_n , \xi_n) = (1+X_n)^\nu$
for some $\nu \in (0,\alpha)$. Then condition~\eqref{eq:AIM2a} follows from~\eqref{eq:D-C}, which also shows that
for $r \in (1, \alpha / \nu)$,
\begin{align*} \Exp [ Y_{n+1}^r - Y_n^r \mid (X_n, \xi_n) = (x,i) ]  
= \frac{c_i}{2} x^{r \nu - \alpha} C^{\rm sym} (r \nu, \alpha ) + o (x^{r \nu -\alpha} )  , \end{align*}
so that  $ \Exp [ Y_{n+1}^r - Y_n^r \mid  \cF_n ] \leq c_i C^{\rm sym} (r \nu, \alpha ) Y_n^{r - (\alpha / \nu)}$,
for all $Y_n$ sufficiently large. Since $\alpha / \nu > 1$ condition~\eqref{eq:AIM2b} follows. Finally,
we may choose $\nu < \alpha$ close enough to $\alpha$ and then
take $\gamma \in (\alpha -1, \nu)$ so that from~\eqref{eq:D-C},
with Lemmas~\ref{lem:C-properties} and~\ref{lem:nu0},
$\Exp [ Y_{n+1}^{\gamma/\nu} - Y_n^{\gamma/\nu} \mid \cF_n ] \geq 0$,
for all $Y_n$ sufficiently large. Thus we may apply Lemma~\ref{lem:AIM2}
to obtain $\Exp [ \sigma^p ] = \infty$ for $p > \gamma / \nu$,
and taking $\gamma$ close to $\alpha-1$ and $\nu$ close to $\alpha$
we can achieve any $p > 1 - (1/\alpha)$, as claimed.

Next suppose that $\alpha \geq 2$.   A similar argument to before but this time using Lemmas~\ref{lem:all-integrals} and~\ref{lem:big-alpha}
shows that for any $\nu \in (0,1)$,
\[ \Exp [ Y_{n+1} - Y_n \mid \cF_n ] \leq - \eps Y_n^{1-(2/\nu)} , \text{ on } \{ Y_n \geq y_0 \} , \]
for some $\eps>0$ and $y_0 \in \RP$. Lemma~\ref{lem:AIM} then shows that
$\Exp [ \sigma^p ] < \infty$ for $p < \nu / 2$ and since $\nu \in (0,1)$ was arbitrary we get
$\Exp [ \sigma^p ] < \infty$ for $p < 1/2$.

We sketch the argument for $\Exp [ \sigma^{p} ] = \infty$ when $p > 1/2$. 
For $\nu \in (1,\alpha)$, it is
not hard to show that $D f_\nu (x,i) \geq 0$ for all $x$ sufficiently large,
and we may verify the other conditions of Lemma~\ref{lem:AIM2} to show that 
$\Exp [ \sigma^p ] = \infty$ for $p > 1/2$.
\end{proof}

\begin{proof}[Proof of Theorem~\ref{thm:one-sided_moments}.]
Most of the proof is similar to that of Theorem~\ref{thm:symmetric_random_walk_moments}, so we omit the details.
The case where a different argument is required is the non-existence part of the case $\alpha \geq 1$.
We have that for some $\eps >0$ and all $y$ sufficiently large, 
$\Pr_{x,i} [ X_1 \geq y ] \geq \eps (x + y)^{-\alpha}$. A similar argument to Lemma~\ref{lem:big-alpha} shows that for any $\nu \in (0,1)$,
for some $C \in \RP$, $\Exp [ X_{n+1}^\nu - X_n^\nu \mid X_n = x ] \geq - C$.
Then a suitable maximal inequality implies that with probability at least $1/2$,
started from $X_1 \geq y$ it takes at least $c y^{\nu}$ steps for $X_n$
to return to a neighbourhood of $0$, for some $c>0$. Combining the two estimates gives
\[ \Pr_{x,i} [ \sigma \geq c y^{\nu} ] \geq \frac{1}{2}(x + y)^{-\alpha} ,\]
which implies $\Exp [ \sigma^p ] =\infty$ for $p \geq \alpha / \nu$, and since $\nu \in (0,1)$
was arbitrary we can achieve any $p > \alpha$.
\end{proof}

\section{Recurrence classification in the critical cases}
\label{sec:critical}

\subsection{Logarithmic Lyapunov functions}

In this section we prove Theorem~\ref{thm:recurrence}(b)(iii).
Throughout this section we write $a_k :=   \cot ( \chi_k \pi \alpha_k )$ and suppose
that $\max_{k \in \cS} \chi_k \alpha_k < 1$, that $\sum_{k \in \cS} \mu_k a_k = 0$, and that 
  $v_i \in \fD^+_{\alpha_i,c_i}$ for all $i \in \cS$, that is,
for $ y >0$,
$v_i ( y ) = c_i (y) y^{-1-\alpha_i}$, with $\alpha_i \in (0,\infty)$
and   $c_i (y) = c_i + O (y ^{-\delta} )$,
where $\delta >0$ may be chosen so as not to depend upon $i$.

To prove recurrence in the critical cases, we need a function that grows
more slowly than any power; now the weights $\lambda_k$ are
additive rather than multiplicative. For $x \in \R$ write $g(x) := \log ( 1 + |x| )$.
Then, for $x \in \RP$ and $k \in \cS$, define 
\begin{equation}
\label{gdef}
 g (x, k) := g (x) + \lambda_k =  \log (1 + |x| ) + \lambda_k , \end{equation}
where $\lambda_k >0$ for all $k \in \cS$.
Also write
\[ h (x , k ) := \left( g (x , k) \right)^{1/2} = \left( \log ( 1 +|x| ) + \lambda_k \right)^{1/2} .\]
Lemma~\ref{lem:finding-theta} shows that there exist $\lambda_k > 0$ ($k \in \cS$) such that
\begin{equation}
\label{eq:crit-a}
 a_k +  \sum_{j \in \cS} p (i, j) ( \lambda_j - \lambda_i ) = 0 ;
\end{equation}
we fix such a choice of the $\lambda_k$ from now on.

We prove recurrence by establishing the following result.

\begin{lemma}
\label{lem:h-lem}
Suppose that the conditions of Theorem~\ref{thm:recurrence}(b)(iii) hold, and that $(\lambda_k ; k \in \cS)$
are such that~\eqref{eq:crit-a} holds. Then there exists $x_0 \in \RP$ such that
\[ \Exp [ h (X_{n+1}, \xi_{n+1} ) - h (X_n, \xi_n ) \mid X_n = x, \xi_n = i ] \leq 0 , ~\text{for} ~ x \geq x_0 \text{ and all } i \in \cS .\]
\end{lemma}

It is not easy to perform the integrals required to estimate $Dh (x, i)$ (and hence establish Lemma~\ref{lem:h-lem})
directly. However, the integrals for $D g (x, i)$
are computable (they appear in Lemma~\ref{lem:j-integrals}), and we can use some analysis to relate
$Dh (x, i)$ to $D g (x, i)$. Thus the first step in our proof of Lemma~\ref{lem:h-lem} is to estimate $D g ( x,i)$.

 For the Lyapunov function $g$, 
\eqref{general-function} gives for $x \in \RP$ and $i \in \cS$,
\begin{align}
\label{additive-function}
Dg(x,i) &= \Exp \left[ g (X_{n+1}, \xi_{n+1}) - g  (X_n,\xi_n) \mid (X_n, \xi_n ) = (x, i) \right] \nonumber\\
& = \sum_{j \in \cS} p (i,j) (\lambda_j - \lambda_i )   \int_{-\infty}^{-x}   w_i (y) \ud y 
+  \int_{-\infty}^\infty   \left( g(x+y) - g(x) \right) w_i (y) \ud y \nonumber \\
& = \chi_i \sum_{j \in \cS} p (i,j) (\lambda_j - \lambda_i )T_{i} (x) 
+  G^\flat_{i} (x),
\end{align}
where we have used
 the fact that $g(x)$ is defined for all $x \in \R$ and 
 symmetric about $0$, and we have introduced the notation
 \[ T_i (x) := \int_x^\infty v_i (y) \ud y = \frac{c_i}{\alpha_i} x^{-\alpha_i} + O (x ^{-\alpha_i - \delta} ), \]
 and 
 \begin{equation}
\label{G-def}
 G^\flat_{i} (x) :=
 \begin{cases}
\frac{1}{2} \int_{-\infty}^\infty \left( g (x+y) - g (x) \right) v_i (|y|) \ud y  
 & \text{ if } \flat={\rm sym}\\
 \int_{0}^\infty \left( g(x-y) - g(x) \right) v_i (y) \ud y  & \text{ if } \flat={\rm one.}
 \end{cases}
 \end{equation}
The next lemma, proved in the next subsection, estimates the integrals in~\eqref{G-def}.

\begin{lemma}
\label{lem:G-estimate}
Suppose that $v_i \in \fD^+_{\alpha_i,c_i}$. Then, for $\flat \in \{ \textup{one}, \textup{sym} \}$, 
for $\alpha \in (0, 1/\chi_i)$,
  for some $\eta >0$, as $x \to \infty$,
 \[ G^\flat_{i} ( x ) =  \chi_i c_{i} x^{-{\alpha_i}} \frac{\pi}{{\alpha_i}} \cot (   \chi_i \pi \alpha_i   ) 
+ O ( x^{-{\alpha_i}-\eta}   ) .\]
\end{lemma}
Note that Lemma~\ref{lem:G-estimate} together with~\eqref{additive-function} shows that
\begin{align}
\label{eq:g-critical}
 D g (x,i)  
& = \frac{\chi_i c_i}{\alpha_i}  \left( a_i + \sum_{j \in \cS} p (i,j) (\lambda_j - \lambda_i ) \right) x^{-{\alpha_i}} + O ( x^{-{\alpha_i}-\eta} ) 
= O ( x^{-\alpha_i -\eta} ) ,
\end{align}
by~\eqref{eq:crit-a}. This is not enough by itself to establish recurrence, since the sign of the $O ( x^{-\alpha_i -\eta} )$ term
is unknown. This is why we need the function $h(x,i)$.

\subsection{Proof of recurrence in the critical case}

\begin{proof}[Proof of Lemma~\ref{lem:G-estimate}.]
To ease notation, we drop the subscripts $i$ everywhere for the duration of this proof.
From \eqref{G-def} we obtain, for $x >0$,
\begin{align}
\label{logint}
\Gs (x  ) & = \frac{1}{2} \int_x^\infty \log \left( \frac{1+y -x}{1+x} \right) v  (y) \ud y \nonumber\\
& {} \quad {} + \frac{1}{2} \int_x^\infty \log \left( \frac{1+x+y}{1+x} \right) v  (y) \ud y \nonumber\\
& {} \quad {} + \frac{1}{2} \int_0^x \left[ \log \left( \frac{1+x+y}{1+x} \right) + \log \left( \frac{1+x-y}{1+x} \right) \right ] v  (y) \ud y .
\end{align}
Let $\alpha \in (0,2)$.
The claim in the lemma for $\flat = \textup{sym}$ will follow from the estimates
\begin{align}
\label{log-integrals}
\int_x^\infty \log \left( \frac{1+x+y}{1+x} \right) v  (y) \ud y & = c x^{-\alpha} j_0^\alpha + O (x^{-\alpha-\eta} ) ; \nonumber\\
\int_x^\infty \log \left( \frac{1+y -x}{1+x} \right) v  (y) \ud y & = c x^{-\alpha} j_2^\alpha + O (x^{-\alpha-\eta} ) ; \nonumber\\
\int_0^x \left[ \log \left( \frac{1+x+y}{1+x} \right) + \log \left( \frac{1+x-y}{1+x} \right) \right ] v  (y) \ud y & = c x^{-\alpha} j_1^\alpha + O (x^{-\alpha-\eta} );
\end{align}
since then
we obtain from~\eqref{logint} with~\eqref{log-integrals} and Lemma~\ref{lem:j-integrals} that
\[ \Gs (x  ) = \frac{c}{2} x^{-\alpha} \frac{1}{\alpha} \left( \psi \left( 1 - \tfrac{\alpha}{2} \right) - \psi \left( \tfrac{\alpha}{2} \right) \right)  + O (x^{-\alpha-\eta} ),\]
which yields the stated result via the digamma reflection formula (equation 6.3.7 from \cite[p.\ 259]{as}).
We present here in detail the proof of only the final estimate in~\eqref{log-integrals}; the others are similar.
Some algebra followed by the substitution $u = y/(1+x)$ 
shows that the third integral in~\eqref{log-integrals} is
\[  \int_0^x \log \left( 1 - \frac{y^2}{(1+x)^2} \right) c (y) y^{-1-\alpha} \ud y 
=   (1+x)^{-\alpha} \int_0^{\frac{x}{1+x}} \log ( 1 -u^2 ) c ( u (1+x) ) u^{-1-\alpha} \ud u .\]
 There is a constant $C \in \RP$ such that for all $x$ sufficiently large,
\[ \int_0^{\frac{\sqrt{x}}{1+x}} \left| \log ( 1 -u^2 ) c ( u (1+x) ) u^{-1-\alpha} \right| \ud u
\leq C  \int_0^{\frac{\sqrt{x}}{1+x}} u^{1-\alpha} \ud u = O ( x^{(\alpha/2)-1} ), \]
using Taylor's theorem for $\log$ and the fact that $c( y)$ is uniformly bounded.
On the other hand,
for $u > \frac{\sqrt{x}}{1+x}$ we have $ c ( u (1+x) ) = c + O (x^{-\delta/2} )$, so that
\[ \int_{\frac{\sqrt{x}}{1+x}}^{\frac{x}{1+x}} \log ( 1 -u^2 ) c ( u (1+x) ) u^{-1-\alpha} \ud u 
= c \int_{\frac{\sqrt{x}}{1+x}}^{\frac{x}{1+x}} \log ( 1 -u^2 ) u^{-1-\alpha} \ud u + O (x^{-\delta/2} ) .\]
Here we have that
\[ \int_{\frac{\sqrt{x}}{1+x}}^{\frac{x}{1+x}} \log ( 1 -u^2 ) u^{-1-\alpha} \ud u 
= j_1^\alpha + O ( x^{(\alpha/2)-1 } ) + O ( x^{-1} \log x ) .\]
Combining these estimates and using the fact that  $\alpha \in (0,2)$,
we obtain the final estimate in~\eqref{log-integrals}.
The claim in the lemma for $\flat = \textup{one}$ follows after some analogous computations, which we omit.
\end{proof}
 
Now we relate $D h (x,i)$ to $D g (x,i)$, by comparing the individual integral terms. 
\begin{lemma}
\label{lem:h-vs-g}
Suppose that $v_i \in \fD^+_{\alpha_i,c_i}$. Then, for all $x$ and all $i$,
\begin{align*}
\int_x^\infty \bigl( h (x+y, i) -h(x,i) \bigr) v_i (y) \ud y & \leq \frac{1}{2 h(x,i)} \int_x^\infty \left( g (x+y ) -g(x ) \right) v_i (y) \ud y ; \\
\int_0^x \bigl( h (x-y ,i) -  h(x,i)  \bigr) v_i (y) \ud y & \leq  \frac{1}{2 h (x,i)} \int_0^x \left( g (x-y) - g(x) \right) v_i (y) \ud y ; \\
\text{and }
 \int_0^x \bigl( h (x+y,i) + h(x-y,i)  &  -  2 h (x, i ) \bigr) v_i (y) \ud y \\
  \leq \frac{1}{2 h (x,i)} \int_0^x \bigl( g(x+y)  &  +  g (x-y) - 2 g(x) \bigr) v_i (y) \ud y .
\end{align*}
Finally, there exists $\eps>0$ such that, for all $i, j \in \cS$ and all $x$ sufficiently large,
\begin{align*}
\int_x^\infty \left( h (y-x, j) - h (x, i) \right) v_i (y) \ud y & \leq \frac{1}{2 h (x,i)} \int_x^\infty \left( g (y-x,j) - g(x,i) \right) v_i (y) \ud y \\
& \qquad\qquad\qquad\qquad {} - \eps \frac{x^{-\alpha_i}}{\log^{3/2} x} .\end{align*} 
\end{lemma}
\begin{proof}
The proof is based on the observation that, since $( h(x,i) )^2 = g(x,i)$, 
\[ h(z,j) - h(x,i) = \frac{g(z,j) - g (x,i)}{h(z,j) + h (x, i)} .\]
Thus, for $y \geq 0$, 
\begin{align}
\label{eq55}
  h (x+y, i) -h(x,i)  = \frac{g(x+y) - g (x)}{h(x+y,i) + h (x, i)} \leq \frac{g (x+y) - g(x)}{2 h (x, i)} ,\end{align}
since $g(x+y) - g (x) \geq 0$ and $h(x+y,i) \geq h(x,i)$. This gives the first inequality in the lemma. Similarly,
for $y \in [0,x]$,
\begin{align*} 
h (x-y ,i) -  h(x,i) = \frac{g (x -y) - g(x)}{h(x-y,i)+h(x,i)} \leq \frac{g (x -y) - g(x)}{2h(x,i)} ,\end{align*}
since $g (x -y) - g(x) \leq 0$ and $h(x-y,i) \leq h(x,i)$. This gives the second inequality, and also yields the third
inequality once combined with the $y \in [0,x]$ case of~\eqref{eq55}.

Finally, for $y \geq x$ note that
 \begin{align}
\label{eq:h-tricky-integral}
  h (y-x ,j )  -h (x ,i )  & = \frac{g (y-x ,j ) -g (x ,i ) }
{  h ( y-x , j ) + h (x , i )} .
\end{align}
Also note that, for $y \geq x > 0$, $g (x , i ) = g(x) + \lambda_i$ and
$g (y-x , j )  = g( y-x ) +\lambda_j$, so
\[  g (y -x , j ) -g (x , i ) = \log \left( \re^{\lambda_j-\lambda_i} \frac{1+y-x}{1+x} \right) .\]
So the sign of the expression in \eqref{eq:h-tricky-integral} is non-positive for $y \leq \psi (x) := x -1 + (1+x) \re^{\lambda_i - \lambda_j}$
and non-negative for $y \geq \psi (x)$, and
\begin{equation}
\label{eq:psi-def} 
g (\psi(x) - x, j ) -g (x , i ) = 0 .\end{equation}
By the monotonicity in $y$ of the denominator,  the expression in \eqref{eq:h-tricky-integral} satisfies
\[ \frac{g (y-x ,j ) -g (x ,i ) }
{  h ( y-x , j ) + h (x , i )}  \leq \frac{g (y-x ,j ) -g (x ,i ) }
{  h ( \psi(x) -x , j ) + h (x , i )}  \]
  both for $y \in [0,\psi(x)]$ and for $y \in [\psi(x), \infty)$.
Here
$h ( \psi(x) - x,j)  =  h (x , i )$,
by \eqref{eq:psi-def}. Hence we obtain the bound
\begin{align*} 
 \int_x^\infty
 \left( h (y-x , j )  -h (x , i ) \right) v_i (y) \ud y   
\leq
\int_x^\infty \left( \frac{g (y-x,j ) -g (x,i  ) }
{  2 h ( x , i  ) } \right) v_i (y) \ud y. \end{align*}
To improve on this estimate, suppose that $y \geq K x$ where $K \in \N$
is such that $K x > \psi(x)$. 
Then, using the fact that the numerator in~\eqref{eq:h-tricky-integral} is positive, 
we may choose $K \in \N$ such that for all $j$ and all $y \geq Kx$,
\[  h (y-x ,j )  -h (x ,i ) \leq \frac{g (y-x ,j ) -g (x ,i ) }
{  h ( (K-1) x  , j ) + h (x , i )} \leq
\frac{g (y-x ,j ) -g (x ,i ) }
{  ( \log ( 1 + |x| ) + \lambda_i + 1 )^{1/2} + h (x,i) } 
, \]
for  $K$ sufficiently large (depending on $\max_j \lambda_j$) and and all $x$ sufficiently large. Here
\[  ( \log ( 1 + |x| ) + \lambda_i + 1 )^{1/2} = h (x,i) \left( 1 + \frac{1}{g(x,i)} \right)^{1/2}
= h (x,i) + \frac{1 + o(1)}{2 \log^{1/2} x } .\]
It follows that, for some $\eps >0$, for all $x$ sufficiently large, for $y \geq Kx$,
\[ h (y-x ,j )  -h (x ,i ) \leq \frac{g (y-x ,j ) -g (x ,i ) }{   2 h(x,i)  } 
 - \eps \left( \frac{g (y-x ,j ) -g (x ,i ) }{  \log^{3/2} x } \right) .\]
The final inequality in the lemma now follows since, for all $x$ sufficiently large,
\begin{align*}
 \int_{Kx} ^\infty  \left( g (y-x ,j ) -g (x ,i ) \right) v_i (y) \ud y 
& \geq \frac{c_i}{2} \int_{K x} ^\infty \log \left( \re^{\lambda_j - \lambda_i }\left(  \frac{2 + y }{1+x} - 1 \right) \right) y^{-1-\alpha_i} \ud y \\
& \geq \frac{c_i}{2} x^{-\alpha_i} \int_{K/2}^\infty \log \left( \re^{\lambda_j - \lambda_i } ( u -1 ) \right) u^{-1-\alpha_i} \ud u ,
\end{align*}
where we used the substitution $u = \frac{2+y}{1+x}$. For $K$ sufficiently large the term inside the logarithm is uniformly positive, and the claimed bound follows.
\end{proof}

Now we may complete the proofs of Lemma~\ref{lem:h-lem} and then Theorem~\ref{thm:recurrence}(b)(iii).

\begin{proof}[Proof of Lemma~\ref{lem:h-lem}.]
Lemma~\ref{lem:h-vs-g} together with~\eqref{eq:g-critical} shows that,  
\[ D h ( x, i) \leq \frac{ D g (x,i)}{ 2 h (x, i) } - \eps \frac{ x^{-\alpha_i}}{\log^{3/2} x}  
 \leq - \eps \frac{ x^{-\alpha_i}}{\log^{3/2} x} + O ( x^{-\alpha_i - \eta } ) \leq 0 ,\]
for all $x$ sufficiently large.
\end{proof}

\begin{proof}[Proof of Theorem~\ref{thm:recurrence}(b)(iii).]
Lemma~\ref{lem:h-lem} with Lemma~\ref{l:adapted_rec_R} shows that $\liminf_{n \to \infty} X_n \leq x_0$, a.s.,
and then a similar argument to that in the proof of parts (a) and (b)(i) of Theorem~\ref{thm:recurrence}
shows that $\liminf_{n \to \infty} X_n =0$, a.s.
\end{proof}

\appendix

\section{Technical results}
\label{sec:appendix}

\subsection{Semimartingale results}
\label{sec:criteria}

\begin{lemma}
\label{l:adapted_rec_R}
Let $(X_n, \xi_n)$ be an $\cF_n$-adapted process taking values in $\RP \times \cS$.
Let $f: \RP \times \cS \to \RP$ be such that $\lim_{x \to \infty} f(x,i) =  \infty$ for all $i \in \cS$, and $\Exp f (X_0, \xi_0 ) < \infty$.
Suppose that there exist  $x_0 \in \RP$ and $C < \infty$
for which,
 for all $n \geq 0$,
\begin{align*}
 \Exp [ f(X_{n+1} , \xi_{n+1} ) - f(X_n, \xi_n ) \mid \cF_n ] & \leq 0, \text{ on } \{ X_n > x_0 \},   \as; \\
\Exp [ f(X_{n+1} , \xi_{n+1} ) - f(X_n , \xi_n ) \mid \cF_n ] & \leq C, \text{ on } \{ X_n \leq x_0 \},   \as \end{align*}
 Then 
\[ \Pr \Bigl[ \bigl\{ \limsup_{n \to \infty} X_n < \infty \bigr\} \cup \bigl\{ \liminf_{n \to \infty} X_n \leq x_0 \bigr\} \Bigr] = 1 .\]
\end{lemma}
\begin{proof}
First note that, by hypothesis,
$\Exp   f(X_{1}, \xi_1 ) \leq \Exp  f (X_0 ,\xi_0 ) + C  < \infty$, and, 
iterating this argument, it follows that $\Exp f (X_n, \xi_n) < \infty$ for all $n \geq 0$.

Fix $n \in \ZP$. For $x_0 \in \RP$ in the hypothesis of the lemma, write  
$\lambda = \min \{ m \geq n  : X_m \leq x_0 \}$.
Let $Y_m = f(X_{m\wedge \lambda}, \xi_{m \wedge \lambda} )$. Then $(Y_m, m \geq n)$
is an $(\cF_m , m \geq n )$-adapted  non-negative supermartingale.
Hence, 
by the supermartingale convergence theorem,
there exists~$Y_\infty \in \RP$ such that $\lim_{m \to \infty} Y_m = Y_\infty$, a.s.
In particular,  
\[ \limsup_{m \to \infty} f(X_m, \xi_m) \leq Y_\infty , \text{ on } \{ \lambda = \infty\} . \]
Set $\zeta_i = \sup \{ x \geq 0 : f(x,i) \leq 1+ Y_\infty \}$,
which has $\zeta_i < \infty$ a.s.~since $\lim_{x \to \infty} f(x,i) = \infty$.
Then $\limsup_{m \to \infty} X_m \leq \max_i \zeta_i < \infty$ on $\{ \lambda = \infty \}$. Hence
\[ \Pr \Bigl[ \bigl\{ \limsup_{n \to \infty} X_n < \infty \bigr\} \cup \bigl\{ \inf_{m \geq n  } X_m \leq x_0 \bigr\} \Bigr] = 1 .\]
Since $n  \in \ZP$ was arbitrary, the result follows:
\[ \Pr \Bigl[ \bigl\{ \limsup_{n \to \infty} X_n < \infty \bigr\} \cup \bigcap_{n  \geq 0} \bigl\{ \inf_{m \geq n } X_m \leq x_0 \bigr\} \Bigr] = 1 . \qedhere \]
\end{proof}

\begin{lemma}
\label{l:adapted_good_escape_R}
Let $(X_n, \xi_n)$ be an $\cF_n$-adapted process taking values in $\RP \times \cS$.
Let $f: \RP \times \cS \to \RP$ be such that
$\sup_{x,i} f(x,i) < \infty$ and
$\lim_{x \to \infty} f(x,i) = 0$ for all $i \in \cS$.
Suppose that there exists $x_1 \in \RP$ for which
$\inf_{y \leq x_1} f(y,i) >0$ for all $i$ and 
\[ \Exp [ f(X_{n+1}, \xi_{n+1} ) - f(X_n, \xi_n) \mid \cF_n ] \leq 0, \text{ on } \{X_n \geq x_1 \}, \text{ for all } n \geq 0. \]
Then for any $\eps >0$ there exists $x \in (x_1, \infty)$
  for which, for all $n \geq 0$,
\[ \Pr \Bigl[ \inf_{m \geq n} X_m \geq x_1 \Bigmid \cF_n \Bigr] \geq 1 - \eps, \text{ on } \{ X_n \geq x \} . \]
\end{lemma}
\begin{proof}
Fix $n \in \ZP$. For $x_1 \in \RP$ in the hypothesis of the lemma, write  
$\lambda = \min \{ m \geq n  : X_m \leq x_1 \}$ and set $Y_m = f(X_{m \wedge \lambda}, \xi_{m \wedge \lambda} )$.
Then $(Y_m , m \geq n )$ is an $(\cF_m , m \geq n )$-adapted non-negative 
supermartingale, and so 
 converges a.s.\ as $m \to \infty$ to some $Y_\infty \in \RP$. 
Moreover, by the optional stopping theorem for supermartingales,
\[ Y_{n} \geq 
\Exp [ Y_\infty \mid \cF_{n} ]
\geq 
\Exp [ Y_\infty \1 { \lambda  < \infty } \mid \cF_{n} ] , \as \]
Here we have that, a.s.,
\begin{align*} Y_\infty   \1 { \lambda < \infty } = \lim_{m \to \infty}
Y_m \1 { \lambda < \infty } & = 
f ( X_{ \lambda } , \xi_\lambda )  \1 { \lambda < \infty } \\
& \geq \min_i \inf_{y \leq x_1} f(y,i)  \1 { \lambda < \infty } .\end{align*}
Combining these inequalities we obtain
\[   \min_i \inf_{y \leq x_1} f(y,i) \Pr [  \lambda < \infty \mid \cF_{n} ] \leq Y_{n}, \as \]
In particular, on $\{ X_{n} \geq x  > x_1 \}$, we have $Y_{n} = f( X_{n}, \xi_n )$ and so 
\[  
\min_i \inf_{y \leq x_1} f(y) \Pr [ \lambda < \infty \mid \cF_n ]
\leq f (X_n, \xi_n) \leq \max_i \sup_{y \geq x } f(y, i)  .\]
Since $\lim_{y \to \infty} f(y,i) = 0$
and $\inf_{y \leq x_1} f(y,i) >0$, 
given $\eps >0$ we can choose $x  > x_1$ large enough so that
\[ \frac{ \max_i \sup_{y \geq x } f(y,i)}{\min_i \inf_{y \leq x_1} f(y,i)} < \eps ;\]
 the choice of $x$ depends only on $f$, $x_1$, and $\eps$, and, in particular, does not depend on $n$. 
Then, 
on $\{ X_n \geq x  \}$, 
$\Pr [ \lambda < \infty \mid \cF_n ] < \eps$, as claimed.
\end{proof}
 
\subsection{Proofs of integral computations}
\label{sec:integrals}

\begin{proof}[Proof of Lemma \ref{lem:i-integrals}.]
Consider $i_0^{\nu,\alpha}$. With the change of variable $v = 1/u$, we get
\begin{align*}
\int_0^1 u^{-1-\alpha} (1+u)^\nu \ud u & = \int_0^1 v^{\alpha - \nu - 1} (1+ v)^\nu \ud v \\
& = \frac{\Gamma (\alpha-\nu) \Gamma (1)}{\Gamma (\alpha - \nu +1 )} \HG ( - \nu, \alpha -\nu ;
\alpha-\nu +1 ; -1 ),\end{align*}
by the integral representation for the Gauss hypergeometric function (see equation 15.3.1 of~\cite[p.~558]{as}).
The given expression for $i_0^{\nu,\alpha}$ follows.

The integral $i_{2,0}^{\alpha}$ is trivial. Consider $i_{2,1}^{\nu,\alpha}$.
The substitution $v = 1/u$ gives
\[  \int_1^\infty u^{-1-\alpha} ( u -1)^{\nu}  \ud u = \int_0^1 (1- v)^{\nu} v^{\alpha-\nu-1} \ud v = \frac{\Gamma (1 +\nu) \Gamma (\alpha -\nu)}{\Gamma (1+\alpha)} ,\]
by the integral formula for the Beta function (see equation 6.2.1 of~\cite[p.~258]{as}), provided $\nu > - 1$ and $\alpha - \nu >0$.
Hence we obtain the given expression for $i_{2,1}^{\nu,\alpha}$.

Next consider $i_1^{\nu,\alpha}$. By considering separately
the asymptotics of the integrand as $u \downarrow 0$ and $u \uparrow 1$,
we see that $i_1^{\nu,\alpha}$ is finite provided
 $\alpha \in (0,2)$ and $\nu > -1$.
For $u\in [-1,1]$, we have the  Taylor series expansion $(1+u)^\nu = \sum_{n\geq 0} \frac{u^n}{n!} \frac{\Gamma(\nu+1)}{\Gamma(\nu+1-n)}$. Hence 
\begin{align*} \left( (1+u)^\nu +(1-u)^\nu-2 \right) u^{-1-\alpha} & = 2 \sum_{n\geq 1} \frac{u^{2n-1-\alpha}}{(2n)!} \frac{\Gamma(\nu+1)}{\Gamma(\nu+1-2n)}  .  \end{align*}
Here, the power series for $n \geq 2$ converges normally (hence uniformly) over $|u| \leq 1$. This remark allows interchanging 
 summation and integration  to obtain 
$i_1^{\nu, \alpha} =  W^{\nu,\alpha} (1)$, where for $|z| \leq 1$ we define
\begin{align}
\label{eq:W-def}
W^{\nu,\alpha}(z)&:= 2\sum_{n\geq 1} \frac{z^{n}}{2n-\alpha} \frac{1}{(2n)!} \frac{\Gamma(\nu+1)}{\Gamma(\nu+1-2n)}\\
&=\frac{1}{2-\alpha} \nu(\nu-1)z\left[ 1 + \frac{2-\alpha}{4-\alpha} \frac{1}{4!} (\nu-2)(\nu-3) z \right. \nonumber\\
& {} \qquad {} \left. + \frac{2-\alpha}{6-\alpha} \frac{1}{6!} (\nu-2)(\nu-3)(\nu-4)(\nu-5) z^2+\ldots \right] \nonumber\\
&= \frac{1}{2-\alpha} \nu(\nu-1)z\left[\sum_{n\geq 0} c_n z^n\right], \nonumber
\end{align}
where $c_n= \frac{2-\alpha}{2n+2-\alpha} \frac{1}{(2(n+1))!} (\nu-2)(\nu-3)\cdots (\nu-2n-1)$, for $n\geq 1$ and $c_0=1$.
An elementary computation yields
\[\frac{c_{n+1}}{c_n}= \frac{(n+1-\alpha/2) (n+1-\nu/2) (n+3/2-\nu/2) (n+1)}{ (n+2 -\alpha/2) (n+2) (n+1/2)}\frac{1}{n+1}.\]
Therefore (see \cite[page 10]{Bailey1935} or \cite[Equation 5.81, page 207]{GrahamKnuthPatashnik1994} for a more easily accessible reference), $\sum_{n\geq 0} c_n z^n= \HGG ( 1, 1-\frac{\nu}{2}, 1-\frac{\alpha}{2}, \frac{3-\nu}{2}, 1 ; 2, \frac{3}{2}, 2-\frac{\alpha}{2} ; z )$. Now, the series defining the generalized hypergeometric function $\mbox{}_pF_q ( \beta_1, \ldots, \beta_p ; \gamma_1, \ldots, \gamma_q ; z)$ for $p=q+1$ converges for all $z$ with $|z|<1$; for $z=1$ the series converges for $\sum_{i=1}^q\gamma_i > \sum_{j=1}^p \beta_j$, a condition that reduces to $\nu > -1$ in the present case. Hence
\[i_1^{\nu,\alpha}= \frac{ \nu(\nu-1) }{2-\alpha} \ \HGG ( 1, 1-\tfrac{\nu}{2}, 1-\tfrac{\alpha}{2}, \tfrac{3-\nu}{2}, 1 ; 2, \tfrac{3}{2}, 2-\tfrac{\alpha}{2} ; 1 ).\]

Finally consider $\tilde i^{\nu,\alpha}_1$.  
Due to the singularity at $u=0$, we compute
\[ \int_0^1( (1-u)^{\nu} - 1 ) u^{-1-\alpha} \ud u = \lim_{t \downarrow 0} \int_t^1 ( (1-u)^{\nu} - 1 ) u^{-1-\alpha} \ud u .\]
For the last integral, we get, for $\alpha > 0$ and $t >0$,
\[  \int_t^1 ( (1-u)^{\nu} - 1 ) u^{-1-\alpha} \ud u =  \int_t^1   (1-u)^{\nu}   u^{-1-\alpha} \ud u + \frac{1 -t^{-\alpha}}{\alpha} .\]
With the substitution $v = \frac{1-u}{1-t}$, the last integral becomes
\begin{align*}  \int_t^1   (1-u)^{\nu}   u^{-1-\alpha} \ud u  
& = (1-t)^{1+\nu} \int_0^1 v^{\nu} (1 - (1-t)v )^{-\alpha -1}  \ud v \\
& = \frac{(1-t)^{1+ \nu } }{1+\nu} \HG ( 1+\alpha , 1+\nu ; 2+\nu ; 1- t) ,\end{align*}
provided $\nu > -1$, by the integral representation for the Gauss hypergeometric function (equation 15.3.1 of~\cite[p.~558]{as}).
Now, by equation 15.3.6 from~\cite[p.\ 559]{as},
\begin{align*}
\HG ( 1+\alpha , 1+\nu ; 2+\nu ; 1- t) & = \frac{\Gamma (2+\nu) \Gamma (-\alpha)}{\Gamma (1-\alpha +\nu)}
\HG ( 1+\alpha , 1+\nu ; 1+\alpha ;  t) \\
& {} ~~~ + t^{-\alpha} \frac{1+\nu}{\alpha} \HG ( 1-\alpha + \nu , 1 ; 1+\alpha ; t ) \\
& = \frac{\Gamma (2+\nu) \Gamma (-\alpha)}{\Gamma (1-\alpha + \nu)} + \frac{1+\nu}{\alpha} t^{-\alpha} + O (t^{1-\alpha} ) ,\end{align*}
as $t \downarrow 0$, provided $\alpha \in (0,1)$. Combining these results we get
\begin{align*} \int_t^1 ( (1-u)^{\nu} - 1 ) u^{-1-\alpha} \ud u  = \frac{1}{\alpha} + \frac{ \Gamma (1+\nu) \Gamma (-\alpha)}{\Gamma (1-\alpha + \nu )} + O (t^{1-\alpha} ).\end{align*}
Letting $t \downarrow 0$ and using the fact that $- \alpha \Gamma (-\alpha) = \Gamma (1-\alpha)$ we obtain the given expression for $\tilde i_{1}^{\nu,\alpha}$.
\end{proof}

\begin{proof}[Proof of Lemma~\ref{lem:j-integrals}.]
We appeal to tables of standard integrals (Mellin transforms) from Section 6.4 of \cite{emot1}.
In particular, the given formulae for
$j_0^\alpha$, $j_2^\alpha$, and $\tilde j_1^\alpha$
follow from, respectively, equations 6.4.17, 6.4.20, and 6.4.19  of \cite[pp.\ 315--316]{emot1};
also used are formulas 6.3.7 and 6.3.8 from~\cite[p.\ 259]{as} and the fact that $\psi (1) = -\gamma$.
Lastly, for $j_1^\alpha$ we use the substitution $u^2 = s$ to obtain
\[ j_1^\alpha = \frac{1}{2} \int_0^1 \frac{\log (1-s)}{s^{1+(\alpha /2 )}} \ud s = \frac{1}{2} \tilde j_1^{\alpha/2} . \qedhere \]
\end{proof}

Finally, we need the following elementary fact.

\begin{lemma}
\label{lem:integral_approx}
Let $A \subseteq \RP$ be a Borel set, and let $f$ and $g$ be measurable functions from $A$ to $\R$.
Suppose that there exist constants $g_-$ and $g_+$ with $0 < g_- < g_+ < \infty$ such that
$g_- \leq g(u) \leq g_+$ for all $u \in A$. Then,
\[ \left| \int_A f(u) g(u) \ud u - g_- \int_A f(u) \ud u \right| \leq (g^+ - g^-) \int_A | f(u) | \ud u .\]
\end{lemma}
\begin{proof}
It suffices to suppose that $\int_A | f(u) | \ud u < \infty$. Then $\int_A | f(u) g(u) | \ud u
\leq g_+ \int_A | f(u) | \ud u < \infty$,
and 
\begin{align*}
\left| \int_A f(u) g(u) \ud u - g_- \int_A f(u) \ud u \right| & = \left| \int_A  f(u) ( g(u) - g_- ) \ud u \right| \\
& \leq \int_A | f(u) | ( g (u) - g_- ) \ud u . \qedhere \end{align*}
\end{proof}

\end{document}